\input ./style/arxiv-general.cfg
\documentclass[aop,MSNbibl,dvips]{arximspdf}
\makeatletter
   \@ifpackageloaded{graphicx}{}{\usepackage{graphicx}}
\makeatother

\doi{10.1214/13-AOP845} 
\volume{43}
\issue{3}
\pubyear{2015}
\firstpage{935}
\lastpage{970}

\makeatletter
\newcommand{\rrvert}{\vert}
\newcommand{\llvert}{\vert}
\newcommand{\eqref}[1]{(\ref{#1})}
\def\cal{\mathcal}

\newtheorem{lemma}{Lemma}[section]
\newproclaim{defn}[lemma]{Definition}

\newtheorem{theorem}[lemma]{Theorem}
\newtheorem{proposition}[lemma]{Proposition}
\newtheorem{propositionn}{Proposition}[section]
\newtheorem{corollary}[lemma]{Corollary}
\newproclaim{rem}[lemma]{Remark}
\newproclaim{notn}[lemma]{Notation}
\newproclaim{example}[lemma]{Example}

\def\PP{\mathbb{P}}
\def\ve{{\varepsilon}}
\def\es{\varnothing}
\def\da{\downarrow}
\def\E{\mathbb{E}}
\def\Q{\mathbb{Q}}
\def\R{\mathbb{R}}
\def\N{\mathbb{N}}
\def\Z{\mathbb{Z}}
\def\D{{\Delta}}
\def\t{{\tau}}
\def\k{{\kappa}}
\def\g{{\gamma}}
\def\s{{\sigma}}
\def\l{{\lambda}}
\def\Th{{\Theta}}
\def\cS{{\cal S}}
\def\cL{{\cal L}}
\def\cR{{\cal R}}
\def\cE{{\cal E}}
\def\cF{{\cal F}}
\def\cD{{\cal D}}
\def\cZ{{\cal Z}}
\def\id{\operatorname{id}}
\def\ua{\uparrow}
\def\sse{\subseteq}
\def\sm{\setminus}
\makeatother

\begin{document}
\begin{frontmatter}

\title{Weak convergence of the localized disturbance flow to the
coalescing
Brownian flow}
\runtitle{Convergence of the disturbance flow}

\begin{aug}
\author[A]{\fnms{James} \snm{Norris}\ead[label=e1]{J.R.Norris@statslab.cam.ac.uk}\thanksref{T1}}
\and
\author[B]{\fnms{Amanda} \snm{Turner}\corref{}\ead[label=e2]{a.g.turner@lancaster.ac.uk}}
\thankstext{T1}{Supported by EPSRC Grant EP/103372X/1.}
\address[A]{Statistical Laboratory\\
Centre for Mathematical Sciences\\
University of Cambridge\\
Wilberforce Road\\
Cambridge, CB3 0WB\\
United Kingdom\\
\printead{e1}}
\address[B]{Department of Mathematics and Statistics\\
Lancaster University\\
Lancaster, LA1 4YF\\
United Kingdom\\
\printead{e2}}
\affiliation{University of Cambridge and Lancaster University}
\runauthor{J.~Norris and A.~Turner}
\end{aug}

\received{\smonth{1} \syear{2013}}
\revised{\smonth{3} \syear{2013}}

%
\begin{abstract}
We define a new state-space for the coalescing Brownian flow, also
known as the Brownian web, on the circle.
The elements of this space are families of order-preserving maps of the circle,
depending continuously on two time parameters and having
a certain weak flow property.
The space is equipped with a complete separable metric.
A larger state-space, allowing jumps in time, is also introduced, and
equipped with a Skorokhod-type metric, also complete and separable.
We prove that the coalescing Brownian flow is the weak limit in this
larger space
of a family of flows which evolve by jumps, each jump arising from a
small localized disturbance of the circle.
A local version of this result is also obtained, in which the weak
limit law is that of the coalescing Brownian flow on the line.
Our set-up is well adapted to time-reversal and our weak limit result
provides a new proof of time-reversibility
of the coalescing Brownian flow. We also identify a martingale
associated with the coalescing Brownian flow on the circle
and use this to make a direct calculation of the Laplace transform of
the time to complete coalescence.
\end{abstract}

%

\end{frontmatter}

\section{Introduction}
This paper is a contribution to the theory of stochastic flows in one dimension.
The main result is Theorem~\ref{MAIN}. It establishes weak convergence
of a certain class of discrete-time stochastic flows on the circle,
which we call disturbance flows, to the coalescing Brownian flow.
This is motivated by a surprising connection with a model of Hastings
and Levitov \cite{HL}
for planar aggregation, which is worked out in our companion paper
\cite{NT2}.
In this model, the flow of harmonic measure on the cluster boundary is
a disturbance flow,
and our convergence theorem then shows that the random structure of fingers
in the Hasting--Levitov cluster is well described in the small-particle
limit by the coalescing Brownian flow.

A disturbance flow is a composition of independent and identically
distributed random maps of the circle to itself.
We do not assume that the maps are homeomorphisms, but do require that
they preserve order.
We consider the limit where the maps are close to the identity and are
well localized.
In this limit, we show that the trajectories of points in the flow
converge weakly to coalescing Brownian motions.
Further, we obtain a corresponding result at the level of flows.
In formulating this, we define some new metric spaces,
which we call the continuous weak flow space and the cadlag weak flow
space. These spaces have a number
of convenient properties, which we prove. In particular, the continuous
weak flow space provides a
state-space for the coalescing Brownian flow where its
independent-increment and reversibility properties
are transparently expressed. The cadlag weak flow space provides a good
framework for weak convergence of one-dimensional stochastic flows with jumps.

The coalescing Brownian flow is, loosely speaking, a family of one
dimensional Brownian motions, one for each space--time starting point,
which evolve independently up to collision and coalesce thereafter.
The possibility to identify a precise mathematical object corresponding
to this idea
was shown by Arratia in 1979 in his Ph.D. thesis \cite{A79}. Beginning
with Arratia, and more recently
pursued by Le Jan and Raimond \cite{LJR} and Tsirelson \cite
{Tsirelson}, one line of work has focused on the
possibility to define a family of random measurable functions $(\phi
_{ts}\dvtx s,t\in\R,s\le t)$, having the flow property
\[
\phi_{ts}\circ\phi_{sr}=\phi_{tr},\qquad r\le s\le
t
\]
and such that any finite collection of trajectories $(\phi
_{ts}(x)\dvtx t\ge
s)$ performs coalescing Brownian motions.
It is known that the functions $\phi_{ts}$ cannot be chosen to be
right-continuous (or left-continuous) and this presents an
obstacle in identifying a suitable metrizable state-space. A second
line of work, initiated by Fontes et al.~\cite{FINR}, overcomes this
difficulty by completing the set of trajectories to form a compact set
of continuous paths (for a well-chosen topology on paths). The space
of these compact sets of paths is then complete and separable for the
Hausdorff metric. Depending on exactly which completion
is chosen, this leads to a number of canonical versions of Arratia's
flow, known as Brownian webs.

In this paper, we follow the flow-type picture, but in order to
overcome the problem of having multiple choices for the value of $\phi
_{ts}(x)$ at points of discontinuity, we work instead with the pairs
$\{\phi^-,\phi^+\}$ of left-continuous and right-continuous
modifications of the Arratia flow. This is not far from the
viewpoint of T\'{o}th and Werner \cite{TW}, who however did not address
questions of weak convergence.
In forgetting the values of $\phi$ at jumps, our state-space becomes
less informative about path properties, but more regular.
We are obliged to relax the flow condition to a ``weak flow'' property
that we define in Section~\ref{CBF}, where we also show how to define a
suitable metric on this space. This gives us an alternative state-space
to~\cite{FINR}, where independent increment and time reversibility
properties are, we think, more naturally expressed; indeed,
time-reversal appears as an isometry. Moreover,
we have been able to develop a Skorokhod-type state-space for flows
which evolve by jumps. This then
dispenses with the need to embed jump flows in continuous flows by
interpolation.

We envisage that there are many natural stochastic flow processes,
which have jumps in time for which continuous interpolation may be
problematic. Our topology provides a convenient framework in which to
characterize these processes and study convergence.
A limitation of our framework is that it requires the flows to have
noncrossing trajectories. In addition, in
our formulation, one does not see so clearly the possible varieties of
path. In models where these properties are important, the topology in
\cite{FINR} may be more appropriate, however, this needs to be weighed
against the complications that may arise from the interpolation process.
An early version of some parts of the present paper, along with its
companion paper \cite{NT2}, appeared in \cite{NT}. A discussion on the
relation between our work and the well-established framework from \cite
{FINR} can also be found in this paper.

The paper is organized as follows.
In Section~\ref{LFC}, we introduce disturbance flows, and we prove weak
convergence
for the trajectories from countably many points, in the limit as the
disturbances become small and well-localized.
In Section~\ref{CBF}, we define the continuous weak flow space and show
that it provides a canonical space
for the coalescing Brownian flow.
Section~\ref{CCTI} is a short digression on the distribution of the
time taken for the coalescing Brownian flow on the circle
to coalesce completely.
In Section~\ref{SKOR}, the larger, cadlag weak flow space, of
Skorokhod type,
is introduced.
The convergence of the disturbance flow to the coalescing Brownian flow
is shown in Section~\ref{MR}. In Section~\ref{TR}, we take advantage of
the approximation
by disturbance flows to give a new proof of the time-reversibility of
the coalescing Brownian flow.
We prove in Section~\ref{LL} a local limit for scales intermediate
between the disturbance and the whole circle, the
limit object being the coalescing Brownian flow on the line. The more
technical proofs can be found in the \hyperref[app]{Appendix}, and a
list of notation
is provided at the end of the paper.

\section{The disturbance flow on the circle}\label{LFC}
We introduce a class of random flows on the circle, whose distributions
are invariant under rotations of the circle and under which each point
on the circle performs a random walk.
The flow maps are in general not continuous on the circle but have an
order-preserving property. In a certain asymptotic regime, the
motion of the flow from a countable family of starting points
is shown to converge weakly to a family of
coalescing Brownian motions.

We specify a particular flow by the choice of a nondecreasing,
right-continuous function $f^+\dvtx\R\to\R$ with the following
\emph{degree}
$1$ property\setcounter{footnote}{1}\footnote{These functions can be considered as liftings of
maps from the circle
$\R/\Z$ to itself having an order-preserving property.
In the limiting regime which we consider, the circle map is a
perturbation of the identity map and
our basic map $f^+$ is the unique lifting which is close to the
identity map on~$\R$.}
%
%
\begin{equation}
\label{circdef} f^+(x+1) = f^+(x) + 1,\qquad  x\in\R.
\end{equation}
Denote the set of such functions by $\cR$ and write $\cL$ for the
analogous set of left-continuous functions.
Each $f^+\in\cR$ has a left-continuous modification $f^-\in\cL$, given
by $f^-(x)=\lim_{y\ua x}f^+(y)$.
Write $\cD$ for the set of all pairs $f=\{f^-,f^+\}$.
When $f^+$ is continuous, we also write $f=f^+$ and, generally, we
write $f$ in place of $f^\pm$
in expressions where the choice of left or right-continuous
modification makes no difference to the value.
The sets $\cR$ and $\cL$ are closed under composition, but $\cD$ is not.
In fact, if $f_1,f_2\in\cD$, then $f_2^-\circ f_1^-$ is the
left-continuous modification of $f_2^+\circ f_1^+$
if and only if $f_1$ sends no interval of positive length to a point of
discontinuity of $f_2$.
We say in this case that $f_2\circ f_1\in\cD$, denoting by $f_2\circ
f_1$ the pair $\{f_2^-\circ f_1^-,f_2^+\circ f_1^+\}$.
Write $\tilde f^\pm$ for the periodic functions $\tilde f^\pm
(x)=f^\pm(x)-x$.
Define $\id(x)=x$ and set
\[
\cD^*= \biggl\{f\in\cD\sm\{\id\}\dvtx\int_0^1
\tilde{f}(x) \,dx=0 \biggr\}.
\]
We assume throughout that our basic map $f\in\cD^*$.

Let us suppose we are given a sequence $(\Th_n\dvtx n\in\Z)$ of independent
random variables,
all distributed uniformly on $(0,1]$.
For $f\in\cD^*$ and ${\theta}\in(0,1]$, define $f_{\theta
}(x)=f(x-{\theta})+{\theta}$. Then
define, for $m,n\in\Z$ with $m<n$,
%
%
\begin{equation}
\label{phiintdef} \Phi_{n,m}^\pm=f_{\Th_n}^\pm
\circ\cdots\circ f_{\Th_{m+1}}^\pm.
\end{equation}
Set $\Phi_{n,n}=\id$ for all $n\in\Z$.
Thus, for $l\le m\le n$, we have $\Phi^\pm_{n,l}=\Phi^\pm
_{n,m}\circ\Phi
^\pm_{m,l}$.
Since $f$ can have at most countably many points of discontinuity and
intervals of
constancy, we have $\Phi_{n,m}=\{\Phi_{n,m}^-,\Phi_{n,m}^+\}\in\cD$
almost surely.
We call the function $f$ the \emph{disturbance} and we call $(\Phi
_{n,m}\dvtx m,n\in\Z,m\le n)$
the \emph{discrete disturbance flow}.\footnote{In the case where $f$
is a
homeomorphism, the
restriction of the flow to $m,n\ge0$ can be recovered from the process
$(\Phi_{n,0}\dvtx n\ge0)$.
This is a random walk on the group of homeomorphisms of the circle. The
structure of this
group is a rich area of mathematics. See, for example, \cite
{MR2903755,MR1876932,MR1713340,MR2084334}. The present paper can be
seen as an investigation
of scaling limits for such random walks with small localized steps.
Our conclusion is then that one has to complete
the homeomorphism group to the space of weak flows in order to support
the limit measure,
and then that, within the class we consider, the limit is universal.}
Define $\rho=\rho(f)\in(0,\infty)$ by
%
%
\begin{equation}
\label{rhodef} \rho\int_0^1
\tilde{f}(x)^2 \,dx=1.
\end{equation}
We embed the discrete-time flow in continuous-time using a Poisson
random measure $N$ on $\R$ of intensity $\rho$.
Write $(T_n\dvtx n\in\Z)$ for the ordered sequence of atoms of $N$, labeled
so that $T_0\le0<T_1$. Then, for each
bounded interval $I\sse\R$, set $\Phi_I=\id$ if $N(I)=0$, and
otherwise set
\[
\Phi_I=\Phi_{n,m},
\]
where $T_{m+1}$ and $T_n$ are the smallest and largest atoms of $N$ in $I$.
Write $\Phi=(\Phi_I\dvtx I\sse\R)$ for the family of maps $\Phi_I$
where $I$
ranges over all bounded intervals in $\R$.
We call $\Phi$ the \emph{Poisson disturbance flow with disturbance $f$}.
A second embedding in continuous time, without additional randomness,
will also be considered.
By the \emph{lattice disturbance flow with disturbance $f$,} we mean the
family $(\Phi_I\dvtx I\sse\R)$, where $\Phi_I=\id$ if $\rho I\cap
\Z=\es$ and
otherwise
$\Phi_I=\Phi_{n,m}$ with $m+1$ the smallest integer and $n$ the largest
integer in the interval $\rho I$.
In each embedding, the time-scale has been chosen to normalize the mean
square displacement per unit time.
Unless otherwise mentioned, our discussion refers to the Poisson case,
which is slightly cleaner, but the
variations needed for the lattice case are slight and we shall end up
with the same asymptotic results in both cases.

Write $I=I_1\oplus I_2$ if $I_1, I_2$ and $I$ are intervals with $\sup
I_1=\inf I_2$, $I_1\cap I_2=\es$ and $I_1\cup I_2=I$.
Note that $\Phi$ has the following properties:
%
%
\begin{eqnarray}\label{LFM}
\begin{tabular}{p{300pt}@{}}
$\Phi^+_I(x)$ and $\Phi^-_I(x)$
are random variables for all bounded intervals $I$ and all $x\in\R$,
\end{tabular}\vspace*{-9pt}
\end{eqnarray}
\begin{equation}
\label{LWF} \Phi_I^+=\Phi_{I_2}^+\circ
\Phi_{I_1}^+\quad\mbox{and\quad $\Phi_I^-=\Phi_{I_2}^-
\circ\Phi_{I_1}^-$\qquad whenever $I=I_1\oplus I_2$},
\end{equation}\vspace*{-9pt}
\begin{eqnarray}\label{LPC}
\begin{tabular}{p{300pt}@{}}
for all $t\in\R$ there exists ${\delta}>0$ such
that for all $s
\in(t-{\delta},t)$
and all $u\in(t,t+{\delta})$, $\Phi_{(s,t)}=
\Phi_{(t,u)}=\id$.
\end{tabular}
\end{eqnarray}
For $e=(s,x)\in\R^2$ and $t\in[s,\infty)$, set
\[
X_t^{e,\pm}=\Phi^\pm_{(s,t]}(x).
\]
For each $e$, almost surely,
%
%
\begin{equation}
\label{XPME} X^{e,-}_t=X^{e,+}_t\qquad
\mbox{for all } t\ge s.
\end{equation}
We will therefore drop the $\pm$ and write simply $X^e=(X^e_t\dvtx
t\ge s)$.
We call $X^e$ the \emph{trajectory of the flow starting from $e$}.
The $\pm$ will reappear in any statement requiring specification of a
version of $X^e$ for uncountably many $e$.
Write $\mu^f_e$ for the distribution of $X^e$ on the Skorokhod space
$D_e=D_x([s,\infty),\R)$ of cadlag paths starting from $x$ at time $s$.
Write $d_e$ for the Skorokhod metric on $D_e$ and
write $\mu_e$ for the distribution on $D_e$ of a standard Brownian
motion starting from $e$.

%
\begin{proposition}\label{LBM}
The trajectory $X^e$ of the Poisson disturbance flow with disturbance
$f$ converges
weakly to Brownian motion on $D_e$, uniformly in $f\in\cD^*$ as $\rho
(f)\to\infty$.
\end{proposition}

\begin{pf}
Write $X$ for $X^e$ within the proof to lighten the notation.
Note that $X$ is a compound Poisson process, making jumps distributed
as $\tilde f(\Th_1)$ at rate $\rho$.
So, for $t\ge s$,
\begin{eqnarray*}
\E(X_t-X_s)&=&\rho(t-s)\int_0^1
\tilde f({\theta})\,d{\theta}=0,
\\
\E\bigl((X_t-X_s)^2\bigr)&=&\rho(t-s)\int
_0^1\tilde f({\theta})^2\,d{\theta}=t-s.
\end{eqnarray*}
Hence, the processes $(X_t)_{t\ge s}$ and $(X_t^2-t)_{t\ge s}$ are martingales.
A standard criterion (see, e.g., \cite{B}, page 143 or \cite
{MR1943877}, page
355) allows us
to deduce that the family of laws $\{\mu^f_e\dvtx f\in\cD^*\}$ is
tight in $D_e$.
Now $f$ is nondecreasing so
\[
\tilde f({\theta})\ge\tilde f({\theta}_0)-({\theta}-{\theta}_0),\qquad
{\theta}\ge
{\theta}_0
\]
and so, if $\tilde f({\theta}_0)\ge0$ for some ${\theta}_0$, then
\[
\rho^{-1}=\int_0^1\tilde f(
{\theta})^2\,d{\theta}\ge\int_{{\theta}_0}^{{\theta}_0+\tilde
f({\theta}
_0)}\bigl(\tilde
f({\theta}_0)-({\theta}-{\theta}_0)\bigr)^2\,d{\theta}=\bigl|\tilde f(
{\theta}_0)\bigr|^3/3
\]
and a similar argument leads to the same estimate also when $\tilde
f({\theta}_0)\le0$.
Hence,
%
%
\begin{equation}
\label{FR3} \bigl|\tilde f({\theta})\bigr|\le(3/\rho)^{1/3},\qquad {\theta}\in(0,1].
\end{equation}
So the jumps of $(X_t)_{t\ge s}$ are bounded in absolute value by
$(3/\rho)^{1/3}$.
Let $\mu$ be any weak limit law for the limit $\rho(f)\to\infty$.
Write $(Z_t)_{t\ge s}$ for the coordinate process on $D_e$.
Then, by standard arguments, $\mu$ is supported on continuous paths and
under $\mu$
both $(Z_t)_{t\ge s}$ and $(Z_t^2-t)_{t\ge s}$
are local martingales in the natural filtration of $(Z_t)_{t\ge s}$.
Hence $\mu=\mu_e$ by L\'evy's characterization of Brownian motion.
\end{pf}

Given a sequence $E=(e_k\dvtx k\in\N)$ in $\R^2$, set
\[
D_E=\prod_{k=1}^\infty
D_{e_k}
\]
and define a metric $d_E$ on $D_E$ by
%
%
\begin{equation}
\label{DEmetric}\quad d_E\bigl(z,z'\bigr)=\sum
_{k=1}^\infty2^{-k}\bigl(d_{e_k}
\bigl(z_k,z'_k\bigr)\wedge1\bigr),\qquad
z=(z_k\dvtx k\in\N), z'=\bigl(z_k'
\dvtx k\in\N\bigr).
\end{equation}
Then $(D_E,d_E)$ is a complete separable metric space and
$(X^{e_k}\dvtx k\in
\N)$ is a random variable in $D_E$.
Write $\mu^f_E$ for the distribution of $(X^{e_k}\dvtx k\in\N)$ on $D_E$.

Write $e_k=(s_k,x_k)$ and denote by $(Z^k_t)_{t\ge s_k}$ the $k$th
coordinate process on $D_E$, given by $Z_t^k(z)=z^k_t$.
Consider the filtration $(\cZ_t)_{t\in\R}$ on $D_E$,
where $\cZ_t$ is the $\s$-algebra generated by $(Z^k_s\dvtx s_k<s\le
t\vee
s_k, k\in\N)$.
Write $C_E$ for the (measurable) subset of $D_E$ where each coordinate
path is continuous.
Define on $C_E$
\[
T^{jk}=\inf\bigl\{t\ge s_j\vee s_k\dvtx
Z_t^j-Z_t^k\in\Z\bigr\}.
\]
We sometimes think of the paths $(Z^k_t)_{t\ge s_k}$ as liftings of
paths in the circle $\R/\Z$. Then
the times $T^{jk}$ are collision times of the circle-valued paths.
The following is a variant of a result of Arratia \cite{A79}.
It provides a useful martingale characterization corresponding to the
intuitive idea of coalescing Brownian motions on the circle.

%
\begin{proposition}\label{ARRA}
There exists a unique Borel probability measure $\mu_E$
on $D_E$ under which, for all $j,k$, the processes
$(Z^k_t)_{t\ge s_k}$ and $(Z^j_tZ^k_t-(t-T^{jk})^+)_{t\ge s_j\vee s_k}$ are
both continuous local martingales in the filtration $(\cZ_t)_{t\in\R}$.
\end{proposition}

We sketch a proof. For existence, one can take independent
Brownian motions from each of the given time--space starting points
and then impose a rule of coalescence on collision, deleting the path
of lower index.
The law of the resulting process has the desired properties. On the
other hand, given a probability
measure such as described in the proposition, on some larger
probability space, one can use a supply of independent Brownian motions
to resurrect the paths deleted at each collision. Then L\'evy's
characterization can be used to see that one has recovered
the set-up used for existence. This gives uniqueness.

Consider now a limit in which the basic map $f$ is an increasingly well
localized perturbation of
the identity, where we quantify this property in terms of the smallest constant
$\l=\l(f)\in(0,1]$ such that
%
%
\begin{equation}
\label{lambdadef} \rho\int_0^1\bigl|\tilde f(x+a)
\tilde f(x)\bigr|\,dx\le\l,\qquad  a\in[\l,1-\l].
\end{equation}

%
\begin{proposition}\label{LAF}
The joint distribution $\mu_E^f$ of the family of trajectories
$(X^e\dvtx e\in E)$ in the Poisson disturbance flow with disturbance $f$
converges weakly to the coalescing Brownian law $\mu_E$
on $D_E$, uniformly in $f\in\cD^*$, as $\rho(f)\to\infty$ and $\l
(f)\to0$.
\end{proposition}

\begin{pf}
We write $X^k$ for $X^{e_k}$ within the proof.
For each $k$, the family of marginal laws $\{\mu^f_{e_k}\dvtx f\in\cD
^*\}$
is tight, as
in Proposition~\ref{LBM}.
Hence, the family of laws $\{\mu^f_E\dvtx f\in\cD^*\}$ is also tight.
Let $\mu$ be any weak limit law for $\{\mu^f_E\dvtx f\in\cD^*\}$
under the
limits $\rho=\rho(f)\to\infty$ and $\l=\l(f)\to0$. Then $\mu$ is
supported on $C_E$.
For all $j,k$ the process
\[
X^j_tX^k_t-\int
_{s_j\vee s_k}^tb\bigl(X^j_s,X^k_s
\bigr)\,ds, \qquad t\ge s_j\vee s_k,
\]
is a martingale,\footnote{In the lattice case, a similar argument can
be based on the martingale
\[
X_t^jX_t^k-\frac{1}\rho
\sum_{n=\lfloor\rho(s_j\vee s_k)\rfloor}^{\lfloor
\rho t\rfloor-1} b\bigl(X^j_{n/\rho},X^k_{n/\rho}
\bigr),\qquad t\ge s_j\vee s_k,
\]
}
where
\[
b\bigl(x,x'\bigr)=\rho\int_0^1
\tilde f(x-{\theta})\tilde f\bigl(x'-{\theta}\bigr)\,d{\theta}.
\]
We have $|b(x,x')|\le\l$ whenever $\l\le|x-x'|\le1-\l$. Hence, by standard
arguments, under $\mu$, the process $(Z^j_tZ^k_t\dvtx s_j\vee s_k\le
t<T^{jk})$ is a local
martingale. We know from the proof of Proposition~\ref{LBM} that, under
$\mu$, the processes
$(Z^j_t\dvtx t\ge s_j)$, $((Z^j_t)^2-t\dvtx t\ge s_j)$ and $(Z^k_t\dvtx
t\ge s_k)$ are
continuous local martingales.
But $\mu$ inherits from the laws $\mu^f_E$ the property
that, almost surely, for all $n\in\Z$, the process
$(Z^j_t-Z^k_t+n\dvtx t\ge
s_j\vee s_k)$ does not change sign.
Hence, by an optional stopping argument, $Z^j_t-Z^k_t$ is constant for
$t\ge T^{jk}$.
It follows that $(Z^j_tZ^k_t-(t-T^{jk})^+)_{t\ge s_j\vee s_k}$ is a
continuous local martingale.
Hence, $\mu=\mu_E$, by Proposition~\ref{ARRA}.
\end{pf}

\section{A new state-space for the coalescing Brownian flow}
\label{CBF}
The weak convergence result for trajectories, obtained in
Proposition~\ref{LAF}, suggests the possibility
of a deeper result at the level of flows, independent of the choice of
starting points for trajectories.
This would be of interest to understand what statistics of the
disturbance flows, beyond trajectories, have weak limits, for
example, trajectories of the inverse, reverse-time flow.
For such a flow-level result, we first specify a state-space and metric
for the notion of weak convergence, and then identify a limit object,
which we call the coalescing Brownian flow.

%
\begin{figure}[t]

\includegraphics{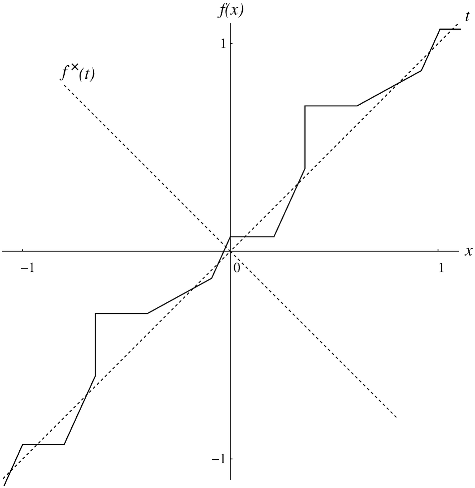}

\caption{The map $f^\times$ obtained from $f$ by rotating
the axes by $\frac{\pi}{4}$.}\label{Phimap}
\end{figure}

We begin by defining a metric on $\cD$.
Let $\cS$ denote the set of all periodic contractions on $\mathbb{R}$
having period 1.
Each $f\in\cD$ can be identified with some $f^\times\in\cS$ by
drawing new
axes at an angle $\pi/4$ with the old, and scaling
appropriately. See Figure~\ref{Phimap}.
More formally, since $x+f^+(x)$ is strictly increasing in $x$, there is
for each $t\in\R$
a unique $x\in\R$ such that
%
%
\begin{equation}
\label{fcrossdef} \frac{x+f^-(x)}2\le t\le\frac{x+f^+(x)}2.
\end{equation}
Define $f^\times(t)=t-x$. Note that $\id^\times=0$.
Then \emph{the map $f\mapsto f^\times\dvtx\cD\to\cS$ is a bijection},
so we can define a metric $d_\cD$ on $\cD$ by
%
%
\begin{equation}
\label{distdist} d_\cD(f,g)=\bigl\|f^\times-g^\times\bigr\|=
\sup_{t\in[0,1)}\bigl|f^\times(t)-g^\times(t)\bigr|.
\end{equation}
A proof of the italicized assertion is given in the \hyperref
[app]{Appendix}. The same
is true for some
further technical assertions which will be made below, written also in italics.
The metric space $(\cS,\|\cdots\|)$ is complete and locally compact,
so the same is true for $(\cD,d_\cD)$.
An alternative characterization\footnote{Thus, $d_\cD$ is a close relative
of the L\'evy metric sometimes used on the set of distribution
functions for real random
variables. This choice of topology is insensitive to the value of a
function at its jump discontinuities, only keeping track of its left
and right continuous versions. The relationships of such a metric to
the operations of composition and inversion
in $\cD$, which are significant for us, do not appear to have been studied.}
of the metric $d_\cD$ is as follows: \emph{for $f,g\in\cD$ and $\ve>0$,
we have}
\begin{eqnarray}
d_\cD(f,g)\le\ve\quad\iff\quad f^-(x-\ve)-\ve\le g^-(x)\le g^+(x) \le f^+(x+
\ve)+\ve
\nonumber\\
\eqntext{\mbox{for all }x\in\R.}
\end{eqnarray}
We deduce that, for $f,g\in\cD$,
\[
d_\cD(f,g)\le\|f-g\|, \qquad 2d_\cD(f,\id)=\|f-\id\|
\]
and
\begin{eqnarray*}
d_\cD(f,g\circ f)&\le&\|g-\id\|\qquad\mbox{when $g\circ f\in\cD$},
\\
d_\cD(f,f\circ g)&\le&\|g-\id\|\qquad\mbox{when $f\circ g\in\cD$}.
\end{eqnarray*}
Moreover, \emph{for any sequence $(f_n\dvtx n\in\N)$ in $\cD$},
\[
f_n\to f\quad \iff\quad  f_n(x)\to f(x)\qquad\mbox{at every point $x$
where $f$ is continuous}.
\]
Here and below, we write $f_n\to f$ to mean convergence in the metric
$d_\cD$.

We now define our space of flows. We call them weak flows to emphasize
that the
usual flow property may fail at points of spatial discontinuity.
Consider $\phi=(\phi_{ts}\dvtx s,t\in\R,s<t)$, with $\phi_{ts}\in
\cD$ for
all $s,t$.
Say that $\phi$ is a \emph{weak flow} if
%
%
\begin{equation}
\label{WFcont} \phi_{ut}^-\circ\phi_{ts}^-\le
\phi_{us}^-\le\phi_{us}^+\le\phi_{ut}^+\circ
\phi_{ts}^+, \qquad s<t<u.
\end{equation}
Say that $\phi$ is \emph{continuous} if, for all $t\in\R$,
\[
\phi_{ts}\to\id\qquad \mbox{as $s\ua t$},\qquad \phi_{ut}\to\id\qquad\mbox{as $u\da t$}.
\]
Write $C^\circ(\R,\cD)$ for the set of all continuous weak flows.
It will be convenient sometimes to extend a continuous weak flow $\phi$
to the diagonal, which we do
by setting $\phi_{ss}=\id$ for all $s\in\R$.
Then, \emph{for any $\phi\in C^\circ(\R,\cD)$, the map}
%
%
\begin{equation}
\label{PHIC} (s,t)\mapsto\phi_{ts}\dvtx\bigl\{(s,t)\dvtx s\le t
\bigr\}\to\cD
\end{equation}
\emph{is continuous}.

Define, for $\phi,\psi\in C^\circ(\R,\cD)$,
%
%
\begin{equation}
\label{dcdist} d_C(\phi,\psi)=\sum_{n=1}^\infty2^{-n}
\bigl\{d_C^{(n)}(\phi,\psi)\wedge1\bigr\},
\end{equation}
where
%
%
\begin{equation}
\label{dcndist} d_C^{(n)}(\phi,\psi)=\sup
_{s,t\in(-n,n),s<t}d_{\cD}(\phi_{ts},
\psi_{ts}).
\end{equation}
Then $d_C$ is a metric on $C^\circ(\R,\cD)$, \emph{under which
$C^\circ
(\R,\cD)$ is
complete and separable}.
Define, for $e=(s,x)\in\R^2$ and $t\ge s$, \emph{evaluation maps}
$Z^{e,+}_t$ and $Z^{e,-}_t$ on $C^\circ(\R,\cD)$ by
\[
Z^{e,\pm}_t(\phi)=\phi_{ts}^\pm(x).
\]
Then, \emph{for all $\phi\in C^\circ(\R,\cD)$, the maps $t\mapsto
Z^{e,\pm}_t(\phi)\dvtx[s,\infty)\to\R$
are continuous.} So we can consider the \emph{left and right coordinate
processes} $Z^{e,\pm}=(Z_t^{e,\pm}\dvtx t\ge s)$ as $C_e$-valued
random variables
on $C^\circ(\R,\cD)$. Write $Z^e=Z^{e,+}$ to lighten the notation.
Define a $\s$-algebra $\cF$ and a filtration $(\cF_t)_{t\in\R}$ on
$C^\circ(\R,\cD)$ by
\[
\cF=\s\bigl(Z^e_t\dvtx e\in\R^2,t\ge s(e)
\bigr),\qquad \cF_t=\s\bigl(Z^e_r\dvtx e\in
\R^2,r\in(-\infty,t]\cap\bigl[s(e),\infty\bigr)\bigr),
\]
where $s(e)$ is the first component of $e$. Then \emph{$\cF_t$ is
generated by the random variables $Z^e_r$ with $e\in\Q^2$ and $r\in
(-\infty,t]\cap[s(e),\infty)$,
and $\cF$ is the Borel $\s$-algebra of the metric $d_C$.}
Define for $e=(s,x)$ and $e'=(s',x')$ the \emph{collision time}
$T^{ee'}\dvtx C^\circ(\R,\cD)\to[0,\infty]$ by
\[
T^{ee'}(\phi)=\inf\bigl\{t\ge s\vee s'\dvtx
Z^e_t(\phi)-Z^{e'}_t(\phi)\in\Z
\bigr\}.
\]

The following result is a variant, stated in the language of continuous
weak flows, of a result of T\'oth and Werner \cite{TW}, Theorem~2.1,
which itself was a variant of a result of Arratia \cite{A79}. The
characterizing martingale properties may be expressed less formally
as saying that there exists a unique probability measure on $C^\circ
(\R
,\cD)$ under which the left and right coordinate processes $Z^{e,\pm}$
agree almost surely for all $e\in\R^2$
and behave as Brownian motions coalescing on the circle.
We shall give a complete proof, in part because we need most components
of the proof also for our main convergence result,
and in part because our framework leads to some simplifications, for
example in the probabilistic underpinnings contained in
Proposition~\ref{MPI}.
The formulation in terms of continuous weak flows has advantages in
leading to a unique object, with a natural time-reversal invariance
(for which see Section~\ref{TR}), and for the derivation of weak limits
(see Section~\ref{MR}).

%
\begin{theorem}\label{UBPR}
There exists a unique Borel probability measure $\mu_A$ on $C^\circ
(\R
,\cD)$ under which,
for all $e,e'\in\R^2$, the processes $(Z^e_t)_{t\ge s(e)}$ and
$(Z^e_tZ_t^{e'}-(t-T^{ee'})^+)_{t\ge s(e)\vee s(e')}$
are continuous local martingales for $(\cF_t)_{t\in\R}$.
Moreover, for all $e\in\R^2$, we have $Z^{e,+}=Z^{e,-}$ $\mu
_A$-almost surely.
\end{theorem}

\begin{pf}
We first show that there exists a unique probability measure $\mu_A$ on
$C^\circ(\R,\cD)$ under which the above property holds for all $e,e'
\in\Q^2$. This essentially amounts to showing that if we have a family
of coalescing Brownian motions starting from every point in $\Q^2$,
then there exists a unique continuous weak flow under which the motions
of each point in $\Q^2$ are the given coalescing Brownian motions.

Fix an enumeration $E=(e_k\dvtx k\in\N)$ of $\Q^2$.
Define the evaluation map $Z^{E,\pm}\dvtx C^\circ(\R,\cD)\to C_E$ by
$Z^{E,\pm}(\phi)=(Z^{e_k,\pm}(\phi)\dvtx k\in\N)$.
Then, \emph{we have $\cF_t=\{(Z^{E,+})^{-1}(B)\dvtx B\in\cZ_t\}$},
where $(\cZ
_t)_{t \in\R}$ is the filtration on $C_E$ generated by projection
mappings as in Proposition~\ref{ARRA}. Therefore, if $\mu$ is any
probability measure on $C^\circ(\R,\cD)$ with the property that for
all $j,k\in\N$, the processes $(Z^{e_k}_t)_{t\ge s_k}$ and
$(Z^{e_j}_tZ_t^{e_k}-(t-T^{e_je_k})^+)_{t\ge s_j\vee s_k}$ are
continuous local martingales for $(\cF_t)_{t\in\R}$; then by
Proposition~\ref{ARRA}, $\mu\circ(Z^{E,+})^{-1}=\mu_E$.

To show existence and uniqueness, it is therefore sufficient to show
that $Z^{E,+}$ is bijective, or rather that there exists some $\mu
_E$-almost sure subset on which $Z^{E,+}$ is bijective.
Let the images of the evaluation maps be
\[
C^{\circ,\pm}_E=\bigl\{Z^{E,\pm}(\phi)\dvtx\phi\in
C^\circ(\R,\cD)\bigr\}.
\]
Then \emph{the sets $C^{\circ,\pm}_E$ are measurable subsets of $C_E$
with $\mu_E(C^{\circ,\pm}_E)=1$.}
Moreover, \emph{$Z^{E,\pm}$ maps $C^\circ(\R,\cD)$ bijectively to
$C^{\circ,\pm}_E$
and the inverse bijections $C^{\circ,\pm}_E\to C^\circ(\R,\cD)$},
which we denote by $\Phi^{E,\pm}$, \emph{are measurable}.
Write $Z^E$ for $Z^{E,+}$ and $\Phi^E$ for $\Phi^{E,+}$.
Then, on $C_E^{\circ,+}$, for all $j,k\in\N$, we have
\[
Z^{e_k}\circ\Phi^E=Z^k,\qquad T^{e_je_k}\circ
\Phi^E=T^{jk},
\]
where $Z^k$ and $T^{jk}$ are the projections and stopping times from
Proposition~\ref{ARRA}, and for all $t\in\R$ and $B\in\cF_t$ we have
$1_B\circ\Phi^E=1_{B'}$ for some $B'\in\cZ_t$.
Thus, we can uniquely define $\mu_A=\mu_E\circ(\Phi^E)^{-1}$ as required.

To complete the proof, we need to show that $\mu_A$ has the required
properties for any given $e,e'\in\R^2$. Observe that all the assertions
above hold also when $E$ is replaced by the sequence
$E'=(e,e',e_1,e_2,\ldots)$.
We repeat the steps taken to obtain a probability measure $\mu'_A=\mu
_{E'}\circ(\Phi^{E'})^{-1}$ on $C^\circ(\R,\cD)$.
Then, under $\mu'_A$, the processes $(Z^e_t)_{t\ge s(e)}$ and
$(Z^e_tZ_t^{e'}-(t-T^{ee'})^+)_{t\ge s(e)\vee s(e')}$
are continuous local martingales for $(\cF_t)_{t\in\R}$.
But also, under $\mu'_A$,
for all $j,k\in\N$, the processes $(Z^{e_k}_t)_{t\ge s_k}$ and
$(Z^{e_j}_tZ_t^{e_k}-(t-T^{e_je_k})^+)_{t\ge s_j\vee s_k}$
are continuous local martingales for $(\cF_t)_{t\in\R}$, so $\mu
_A=\mu'_A$.

Finally, \emph{we have $\Phi^{E',+}=\Phi^{E',-}$ on $C_{E'}^{\circ
,-}\cap C_{E'}^{\circ,+}$}, so
\[
Z^{e,-}\bigl(\Phi^{E'}\bigr)=Z^{e,-}\bigl(
\Phi^{E',-}\bigr)=Z^{e,+}\bigl(\Phi^{E'}\bigr),
\]
$\mu_{E'}$-almost surely,
and so $Z^{e,-}=Z^{e,+}$, $\mu_A$-almost surely, as claimed.
\end{pf}

We call any $C^\circ(\R,\cD)$-valued random variable with law $\mu
_A$ a
\emph{coalescing Brownian flow on the circle}.

\section{Complete coalescence time}\label{CCTI}
In this section, we digress to discuss the complete coalescence time $T$
of a coalescing Brownian flow $\Phi$ on the circle, given by
\[
T=\inf\bigl\{t\ge0\dvtx\Phi^+_{t0}(x)=y+n\mbox{ for some $n\in\Z$,
for all $x\in\R$, for some $y\in\R$}\bigr\}.
\]
It is known that
%
%
\begin{equation}
\label{CCT} \E\bigl(e^{\l T}\bigr)=\sqrt{\l}/\sin\sqrt{\l},\qquad \l<
\pi^2.
\end{equation}
Cox \cite{MR1048930} showed this by an indirect argument. More
recently, Zhou \cite{Zhou} gave a direct proof.
We give an alternative and simpler proof.

Fix $N\in\N$ and define for $t\ge0$
\[
B^k_t=\Phi_{t0}(k/N)-\Phi_{t0}
\bigl((k-1)/N\bigr),\qquad  k=1,\ldots,N.
\]
Then each process $B^k$ is a Brownian motion of diffusivity $2$,
starting from $1/N$ and stopped on hitting $0$ or $1$.
Consider the stopping time $S=\inf\{t\ge0\dvtx B_t^k=1\mbox{ for
some $k$}\}
$ and note
that $B^k_S=0$ for all but one random value, $k=K$ say, for which $B^K_S=1$.
Define
\[
M_t=M_t^{(N)}=e^{\l t}\sum
_{k=1}^N\sin\bigl\{\sqrt{\l}B_t^k
\bigr\}
\]
then the stopped process $(M_t^S)_{t \geq0} = (M_{S \wedge t})_{t \geq
0}$ is a martingale so, for all $t\ge0$,
\begin{eqnarray*}
N\sin\{\sqrt{\l}/N \}&=&M_0
\\
&=&\E(M_{S\wedge t})
\\
&=&\E\Biggl(e^{\l(S\wedge t)}\sum_{k=1}^N
\sin\bigl\{\sqrt{\l}B^k_{S\wedge t} \bigr\} \Biggr)
\\
&\ge&\E\bigl(e^{\l(S\wedge t)}\bigr)\sin\sqrt{\l}.
\end{eqnarray*}
For $\l<\pi^2$ the final inequality allows us to see that $\E(e^{\l
S})<\infty$, so we can let $t\to\infty$
to obtain
\[
N\sin\{\sqrt{\l}/N\}=\E\bigl(e^{\l S}\bigr)\sin\sqrt{\l}.
\]
On letting $N\to\infty$, we obtain (\ref{CCT}).

In fact, it is not hard to see that $M^{(N)}_t$ increases with $N$ for
all $t\ge0$
and is eventually constant for all $t>0$. The limit process $M^{(\infty
)}$ is also a martingale
with $M^{(\infty)}_0=\sqrt{\l}$ and $M^{(\infty)}_T=e^{\l T}\sin
\sqrt{\l
}$, and the optional stopping
argument can alternatively be applied directly to $M^{(\infty)}$.

From (\ref{CCT}), we can identify $T$ as having the same law as
one-half of the time $\tilde T$ taken
for a BES($3$) to get from $0$ to $1$. This can also be seen directly
using the relation
\[
S=\sum_{k=1}^NS_k1_{\{B^k(S_k)=1\}},
\]
where $S_k=\inf\{t\ge0\dvtx B_t^k\in\{0,1\}\}$. Then, for any bounded
measurable function $f$,
\[
\E\bigl(f(S)\bigr)=\sum_{k=1}^N\E
\bigl(f(S_k)1_{\{B^k(S_k)=1\}}\bigr)=\E\bigl(f(S_1)|B^1(S_1)=1
\bigr)
\]
and, on letting $N\to\infty$, we obtain $\E(f(T))=\E(f(\tilde T/2))$.
We thank Neil O'Connell and Marc Yor
for this observation.

\section{A Skorokhod-type space of nondecreasing flows on the circle}
\label{SKOR}
Since the disturbance flow is not continuous in time, it will be
necessary to
introduce a larger flow space to accommodate it.
Consider now $\phi=(\phi_I\dvtx I\sse\R)$, where $\phi_I\in\cD$
and $I$ ranges
over all nonempty bounded intervals.
Recall that we write $I=I_1\oplus I_2$ if $I, I_1, I_2$ are intervals
with $\sup I_1=\inf I_2$, $I_1\cap I_2=\es$ and $I_1\cup I_2=I$.
Say that $\phi$ is a \emph{weak flow} if
%
%
\begin{equation}
\label{WF} \phi_{I_2}^-\circ\phi_{I_1}^-\le
\phi_I^-\le\phi_I^+\le\phi_{I_2}^+\circ
\phi_{I_1}^+,\qquad I=I_1\oplus I_2.
\end{equation}
Say that $\phi$ is \emph{cadlag}\footnote{This definition is more
symmetric in time than
is usual for ``cadlag'': a more accurate acronym would be \emph{laglad.}}
if, for all $t\in\R$,
\[
\phi_{(s,t)}\to\id \qquad\mbox{as } s\uparrow t,\qquad \phi_{(t,u)}\to
\id\qquad \mbox{as } u\downarrow t.
\]
Write $D^\circ(\R,\cD)$ for the set of cadlag weak flows.
It will be convenient to extend a cadlag weak flow $\phi$ to the empty
interval by setting $\phi_\es=\id$.
Given a bounded interval $I$ and a sequence of bounded intervals
$(I_n\dvtx n\in\N)$, write $I_n\to I$ if the indicator functions $1_{I_n}
\to1_{I}$ pointwise as $n \to\infty$.
\emph{For any $\phi\in D^\circ(\R,\cD)$, we have}
%
%
\begin{equation}
\label{PHID} \phi_{I_n}\to\phi_I \qquad\mbox{as }
I_n\to I.
\end{equation}

Let $\phi$ be a cadlag weak flow and suppose that $\phi_{\{t\}}=\id$
for all $t\in\R$.
Then, using \eqref{WF}, we have $\phi_{(s,t)}=\phi_{(s,t]}=\phi
_{[s,t)}=\phi_{[s,t]}$ for all $s<t$ and, denoting
all these functions by $\phi_{ts}$,\footnote{Note the reversal of the
order of $s$ and $t$. This was chosen to make the weak flow property
\eqref{WFcont} appear neater.} the family $(\phi_{ts}\dvtx s,t\in\R
,s<t)$ is
a continuous weak
flow in the sense of the preceding section.

For $\phi,\psi\in D^\circ(\R,\cD)$ and $n\ge1$, define
%
%
\begin{equation}
\label{ddndist} d_D^{(n)}(\phi, \psi) = \inf
_\lambda\Bigl\{ \gamma(\lambda) \vee\sup_{I\sse\R}
\bigl\|\chi_n(I)\phi_I^\times-\chi_n
\bigl(\l(I)\bigr)\psi_{\lambda(I)}^\times\bigr\| \Bigr\},
\end{equation}
where the infimum is taken over the set of increasing homeomorphisms
$\lambda$ of $\R$, where
%
%
\begin{equation}
\label{gammadef} \gamma(\lambda) = \sup_{t\in\R}\bigl|\l(t)-t\bigr|\vee
\sup
_{s,t\in\R,s<t} \biggl\llvert\log\biggl(\frac{\lambda(t) -
\lambda(s)}{t-s} \biggr)
\biggr\rrvert,
\end{equation}
and where $\chi_n$ is the cutoff function\footnote{As in the case of
the standard Skorokhod topology, localization
in time sits awkwardly with the stretching of time introduced via the
homeomorphisms $\l$. There is no fundamental
obstacle, just some messiness at the edges. Note that, when $I\cup\l
(I)\sse[-n,n]$, we have
\[
\bigl\|\chi_n(I)\phi_I^\times-\chi_n
\bigl(\l(I)\bigr)\psi_{\lambda(I)}^\times\bigr\|=d_\cD(
\phi_I,\psi_{\l(I)}).
\]
Also, for all intervals $I$, we have $|\chi_n(\l(I))-\chi_n(I)|\le
\g(\l
)$ and
\[
\bigl\|\chi_n(I)\phi_I^\times-\chi_n
\bigl(\l(I)\bigr)\psi_{\lambda(I)}^\times\bigr\| \le\chi_n(I)d_\cD(
\phi_I,\psi_{\l(I)})+\bigl|\chi_n\bigl(\l(I)\bigr)-
\chi_n(I)\bigr|\bigl\|\psi^\times_{\l(I)}\bigr\|.
\]
}
given by
\[
\chi_n(I)=0\vee(n+1-R)\wedge1,\qquad R=\sup I\vee(-\inf I).
\]
Then define
%
%
\begin{equation}
\label{SMET} d_D(\phi, \psi)=\sum_{n=1}^\infty2^{-n}
\bigl\{d_D^{(n)}(\phi,\psi)\wedge1\bigr\}.
\end{equation}
Then $d_D$ is a metric on $D^\circ(\R,\cD)$ \emph{under which
$D^\circ(\R
,\cD)$ is
complete and separable.}
Moreover, \emph{the metrics $d_C$ and $d_D$ generate the same topology
on $C^\circ(\R,\cD)$}.
\emph{For the metric $d_D$, for all bounded intervals $I$ and all
$x\in\R
$, the} \textup{evaluation map}
\[
\phi\mapsto\phi_I^+(x)\dvtx D^\circ(\R,\cD)\to\R
\]
\emph{is Borel measurable.
Moreover, the Borel $\s$-algebra on $D^\circ(\R,\cD)$ is generated by
the set of all such evaluation maps with $I=(s,t]$ and $s,t$ and $x$ rational.}

\section{Convergence to the coalescing Brownian flow}
\label{MR}

We now give a criterion for weak convergence on $D^\circ(\R,\cD)$ and
use it to show that the disturbance flow converges to the coalescing
Brownian flow.

\emph{For $e=(s,x)\in\R^2$ and $\phi\in D^\circ(\R,\cD)$, the maps}
\[
t\mapsto\phi_{(s,t]}^\pm(x)\dvtx[s,\infty)\to\R
\]
\emph{are cadlag.} Hence, we can extend the maps $Z^e=Z^{e,+}$ and $Z^{e,-}$,
which we defined on $C^\circ(\R,\cD)$ in
Section~\ref{CBF}, to measurable maps
$Z^{e,\pm}\dvtx D^\circ(\R,\cD)\to D_e$ by setting
\[
Z^{e,\pm}(\phi)=\bigl(\phi_{(s,t]}^\pm(x)\dvtx t\ge s
\bigr).
\]
Let $E=(e_k\dvtx k\in\N)$ be any countable dense subset of $\R^2$.
Write $Z^{E,\pm}$ for the maps $D^\circ(\R,\cD)\to D_E$ given by
$Z^{E,\pm}=(Z^{e_k,\pm}\dvtx k\in\N)$. Write $Z^E=Z^{E,+}$.
The following result is a criterion for weak convergence on $D^\circ
(\R
,\cD)$.
If we restrict to measures supported on $C^\circ(\R,\cD)$, this is directly
analogous to \cite{FN}, Theorem~4.1.

%
\begin{theorem}\label{WCC}
Let $(\mu_n\dvtx n\in\N)$ and $\mu$ be Borel probability measures
on $D^\circ
(\R,\cD)$.
Assume that $Z^{E,-}=Z^{E,+}$ holds $\mu_n$-almost surely for all $n$
and $\mu$-almost surely.
Assume further that $\mu_n\circ(Z^E)^{-1}\to\mu\circ(Z^E)^{-1}$ weakly
on $D_E$.
Then $\mu_n\to\mu$ weakly on $D^\circ(\R,\cD)$.
\end{theorem}

\begin{pf}
Set
\begin{eqnarray*}
D^\circ(E)&=&\bigl\{\phi\in D^\circ(\R,\cD)\dvtx
Z^{E,+}(\phi)=Z^{E,-}(\phi)\bigr\},
\\
D^\circ_E&=&\bigl\{Z^E(\phi)\dvtx\phi\in
D^\circ(E)\bigr\}.
\end{eqnarray*}
Let $\Phi_n$ and $\Phi$ be random variables in $D^\circ(\R,\cD)$ having
distributions $\mu_n$ and $\mu$, respectively.
Then $Z^E(\Phi_n)\to Z^E(\Phi)$ weakly on $D_E$.
Also $\Phi_n, \Phi\in D^\circ(E)$ almost surely, so $Z^E(\Phi_n),
Z^E(\Phi)\in D^\circ_E$ almost surely.
Now \emph{$D^\circ_E$ is measurable and $Z^E$ maps $D^\circ(E)$
bijectively to $D^\circ_E$}.
Denote the inverse bijection by $\Phi^E$.
Then \emph{$\Phi^E\dvtx D_E^\circ\to D^\circ(E)$ is measurable and
continuous}.\vspace*{1pt}
Hence, $\Phi_n=\Phi^E(Z^E(\Phi_n))\to\Phi^E(Z^E(\Phi))=\Phi$
weakly on
$D^\circ(\R,\cD)$.
\end{pf}

The Poisson disturbance flow with disturbance $f$ and the lattice
disturbance flow with disturbance $f$ were defined in Section~\ref{LFC}.
Properties (\ref{LFM}), (\ref{LWF}) and (\ref{LPC}) hold in both
cases and
imply that the flow $\Phi=(\Phi_I\dvtx I\sse\R)$ may be considered
as a
Borel random variable in $D^\circ(\R,\cD)$.
Moreover, as we noted in (\ref{XPME}), for either of these flows $\Phi
$, for all $e\in\R^2$, we have $Z^{e,-}(\Phi)=Z^{e,+}(\Phi)$
almost surely.
The same is true when $\Phi$ is a coalescing Brownian flow, as shown in
Theorem~\ref{UBPR}.
Our main result now follows directly from Proposition~\ref{LAF} and
Theorem~\ref{WCC}.

%
\begin{theorem}\label{MAIN}
The Poisson disturbance flow with disturbance $f$
and the lattice disturbance flow with disturbance $f$
both converge weakly to the coalescing Brownian flow on the circle
on $D^\circ(\R,\cD)$,
uniformly in $f\in\cD^*$ as $f$ becomes small and localized,
that is, as $\rho(f)\to\infty$ and $\l(f)\to0$.
\end{theorem}

\section{Time reversal}\label{TR}
Time reversal acts as an isometry on our metric spaces of weak flows.
The time reversal of a disturbance flow with disturbance $f$ is the
disturbance flow with disturbance $f^{-1}$.
We use these facts to give a new proof of the time-reversibility of the
coalescing Brownian flow, and to
obtain a weak limit for the joint law of forward and backward
trajectories for disturbance flows.

For $f^+\in\cR$ and $f^-\in\cL$, we define a \emph{left-continuous inverse}
$(f^+)^{-1}\in\cL$ and a \emph{right-continuous inverse}
$(f^-)^{-1}\in
\cR$
by
\begin{eqnarray*}
\bigl(f^+\bigr)^{-1}(y)&=&\inf\bigl\{x\in\R\dvtx f^+(x)>y\bigr\},
\\
\bigl(f^-\bigr)^{-1}(y)&=&\sup\bigl\{x\in\R\dvtx f^-(x)<y\bigr\}.
\end{eqnarray*}
The map $f^+\mapsto(f^+)^{-1}\dvtx\cR\to\cL$ is a bijection, with
$((f^+)^{-1})^{-1}=f^+$ and
\[
\bigl(f_1^+\circ f_2^+\bigr)^{-1}=
\bigl(f_2^+\bigr)^{-1}\circ\bigl(f_1^+
\bigr)^{-1}, \qquad f_1,f_2\in\cR.
\]
We have $f^+\circ(f^+)^{-1}=\id$ if and only if $f^+$ is a homeomorphism.
Define for $f=\{f^-,f^+\}\in\cD$ the \emph{inverse}
$f^{-1}=\{(f^+)^{-1},(f^-)^{-1}\}\in\cD$. Note that $(f^{-1})^\times
=-f^\times$,
so the map $f\mapsto f^{-1}\dvtx\cD\to\cD$ is an isometry.
Define the \emph{time-reversal map} $\wedge\dvtx D^\circ(\R,\cD
)\to D^\circ(\R
,\cD)$ by
\[
\hat\phi_I=\phi_{-I}^{-1},
\]
where $-I=\{-x\dvtx x\in I\}$.
It is straightforward to check that this is a well-defined isometry of
$D^\circ(\R,\cD)$, which restricts to an isometry of $C^\circ(\R
,\cD)$.

%
\begin{proposition}\label{TRD}
The time-reversal of a disturbance flow with disturbance $f$ is a
disturbance flow with
disturbance $f^{-1}$.
\end{proposition}

\begin{pf}
Fix $f\in\cD^*$. Set $g=f^{-1}$ and
\begin{eqnarray*}
\D&=&\bigl\{(x,y)\in\R^2\dvtx y<f(x)\bigr\}=\bigl\{(x,y)\in
\R^2\dvtx x>g(y)\bigr\},
\\
\D_0&=&\bigl\{(x,y)\in\R^2\dvtx y<x\bigr\}.
\end{eqnarray*}
Then, by Fubini's theorem,
%
%
\begin{equation}
\label{FM} \int_0^1\tilde f(x)\,dx=\int
_0^1\int_\R(1_\D-1_{\D_0})
(x,y)\,dx\,dy=-\int_0^1\tilde g(y)\,dy
\end{equation}
and
%
%
\begin{equation}
\label{SM} \int_0^1\tilde
f(x)^2\,dx=\int_0^1\int
_\R2(y-x) (1_\D-1_{\D
_0}) (x,y)\,dx\,dy=
\int_0^1\tilde g(y)^2\,dy.
\end{equation}
So $g\in\cD^*$ and $\rho(g)=\rho(f)$.
We may construct a lattice disturbance flow $\Phi$ with disturbance $f$
from a sequence $(\Th_n\dvtx n\in\Z)$
of independent random variables, uniformly distributed on $(0,1]$, by
\[
\Phi_I^\pm=f_{\Th_n}^\pm\circ\cdots
\circ f_{\Th_m}^\pm,
\]
where $m$ and $n$ are respectively the minimal and maximal integers in
$\rho I$.
Then
\[
\hat\Phi^\pm_I=g_{\Th_{-n}}^\pm\circ
\cdots\circ g_{\Th_{-m}}^\pm.
\]
Since $(\Th_n\dvtx n\in\Z)$ and $(\Th_{-n}\dvtx n\in\Z)$ have the same
distribution, it follows that $\hat\Phi$
is a lattice disturbance flow with disturbance $g$. The Poisson case is similar.
\end{pf}

We were surprised by the calculations (\ref{FM}) and (\ref{SM}) which,
though elementary, we did not
suspect until we realized they were forced by the known reversibility
of the universal scaling limit.
On the other hand, we can now deduce the reversibility of the limit, as
already known
for other formulations of the coalescing Brownian flow.
See, for example, \cite{A79,FN,STW,Zhou} and the references therein.

%
\begin{corollary}
\label{timerefcor}
The law $\mu_A$ of the coalescing Brownian flow on the circle is
invariant under time-reversal.
\end{corollary}

\begin{pf}
Fix $r\in(0,1/2]$ and define $f=f_r\in\cD^*$ by
\[
f^+(n+x)=n+\bigl(r\vee x\wedge(1-r)\bigr),\qquad  n\in\Z, x\in[0,1).
\]
Then $\tilde f^+(x)=((r-x)\vee0)+((1-r-x)\wedge0)$ for $x\in[0,1)$,
so $\rho(f)=3/(2r^3)$ and
\[
\int_0^1\tilde f(x)\tilde f(x+a)\,dx=0,\qquad  2r\le
a\le1-2r,
\]
so $\l(f)\le2r$. Moreover, $\rho(f^{-1})=\rho(f)$ and $\l
(f^{-1})\le2r$.

Write $\mu^f_A$ for the law of a lattice disturbance flow with
disturbance $f$.
Set $\hat\mu_A=\mu_A\circ\wedge^{-1}$ and $\hat\mu_A^f=\mu
_A^f\circ
\wedge^{-1}$.
Consider the limit $r\to0$. By Theorem~\ref{MAIN},
we know that $\mu^f_A\to\mu_A$ and $\mu^{f^{-1}}_A\to\mu_A$,
weakly on
$D^\circ(\R,\cD)$.
Since the time-reversal map $\phi\mapsto\hat\phi$ is an isometry, it
follows, using the preceding proposition, that $\mu^{f^{-1}}_A=\hat
\mu
^f_A\to\hat\mu_A$,
weakly on $D^\circ(\R,\cD)$. Hence, $\mu_A=\hat\mu_A$.
\end{pf}

The same argument may be used to prove time reversibility of the
coalescing Brownian flow
on the line, as introduced in the next section.
In fact, Theorem~\ref{MAINT} below applies to show that the $\sqrt
{r}$-scale disturbance flow (defined below) with disturbance $f_r$ (as above)
converges weakly as $r\to0$ to the coalescing Brownian flow on the
line. Then reversibility follows
by the argument of Corollary~\ref{timerefcor}.

From the flow-level result Theorem~\ref{MAIN}, we can deduce weak
convergence also for
paths running forward and backward in time from a given sequence of
points $E=(e_k\dvtx k\in\N)$ in $\R^2$.
For $e=(s,x)\in\R^2$, define $\check D_e=\{\xi\in D(\R,\R)\dvtx
\xi_s=x\}$
and set $\check D_E=\prod_{k=1}^\infty\check D_{e_k}$.
For $\phi\in D^\circ(\R,\cD)$, define
%
%
\begin{equation}
\label{zcheckdef} \check Z^{e,\pm}_t(\phi)= %
\cases{
\phi^\pm_{(s,t]}(x),& \quad $t\ge s$,\vspace*{2pt}
\cr
\bigl(
\phi^{-1}\bigr)^\pm_{(t,s]}(x),& \quad $t<s$.} %
\end{equation}
Then $\check Z^{e,\pm}(\phi)\in\check D_e$ and extends $Z^{e,\pm
}(\phi
)$, as defined in Section~\ref{SKOR}, from $[s,\infty)$ to the whole of
$\R$.
For all $e\in\R^2$, we have $\check Z^{e,+}=\check Z^{e,-}$ almost
everywhere on $D^\circ(\R,\cD)$
for both $\mu_A$ and $\mu_A^f$, for any disturbance $f$. So, we drop
the $\pm$.
Denote by $\check\mu_E^f$ the law of $(\check Z^{e_k}\dvtx k\in\N)$ on
$\check D_E$ under $\mu_A^f$
and by $\check\mu_E$ the corresponding law under $\mu_A$.

%
\begin{corollary}\label{BFP}
We have $\check\mu_E^f\to\check\mu_E$ weakly on $\check D_E$, uniformly
in $f\in\cD^*$, as $\rho(f)\to\infty$ and $\l(f)\to0$.
\end{corollary}

\begin{pf}
We can check that $\check Z^{(s,x),+}$ is continuous as a map $D^\circ
(\R,\cD)\to\check D_{(s,x)}$ at $\phi\in C^\circ(\R,\cD)$ provided
\[
\check Z^{(s,x\pm{\delta}),+}(\phi)\to\check Z^{(s,x),+}(\phi)
\]
uniformly on $\R$ as ${\delta}\to0$. Since this property holds for
$\mu_A$
almost all $\phi$, the claimed limit
follows from Theorem~\ref{MAIN} by a standard property of weak convergence.
\end{pf}

Weak convergence of the forward paths to coalescing Brownian motions
was shown in Proposition~\ref{LAF}.
The corresponding backward property is immediate from the fact that
the time reversal of a disturbance
flow is another such flow. What is new in the result just proved is the
identification of the limit of the joint law of these
backward and forward paths---which has the property that the
bi-infinite paths never cross.

\section{Local limits}\label{LL}
We now prove local weak convergence of disturbance flows, for a scale
$\ve\in(0,1]$ intermediate between the scale
of the disturbance $f$ and the unit scale of the circle.
Some variations of our set-up will be needed, as we rescale in a way
which does not
preserve the degree $1$ property (\ref{circdef}), and the limit object
is the coalescing Brownian flow on the line.
Write $\bar\cD$ for the set of all pairs $\{f^-,f^+\}$ where
$f^+\dvtx\R\to
\R$ is nondecreasing
and right-continuous and where $f^-$ is the left-continuous
modification of $f^+$.
For $\ve\in(0,1]$, define the scaling map $\s_\ve\dvtx\bar\cD\to
\bar\cD$ by
\[
\s_\ve f(x)=\ve^{-1}f(\ve x).
\]
This map can be thought of as zooming in on the neighborhood around the
origin. We associate to a disturbance flow $\Phi=(\Phi_I\dvtx I\sse
\R)$
the \emph{$\ve$-scale disturbance flow} $\Phi^\ve=(\Phi_I^\ve\dvtx
I\sse\R)$,
given by
\[
\Phi_I^\ve=\s_\ve(\Phi_{\ve^2 I}).
\]
For $e\in\R^2$, we write $X^{e,\ve}$ for the trajectory of $\Phi
^\ve$
starting from $e$.
By the estimate~(\ref{FR3}), the jumps of $X^{e,\ve}$ are bounded in
absolute value by $\ve^{-1}(3/\rho)^{1/3}$.
A~small variation of the proof of Proposition~\ref{LBM} then leads to
the following result.

%
\begin{proposition}\label{LBMT}
The trajectory $X^{e,\ve}$ of the $\ve$-scale Poisson disturbance flow
with disturbance $f$ converges
weakly to Brownian motion on $D_e$, uniformly in $f\in\cD^*$ and $\ve
\in
(0,1]$ as $\ve^3\rho(f)\to\infty$.
\end{proposition}

Fix a sequence $E=(e_k\dvtx k\in\N)$ in $\R^2$ and write $D_E$ and
$C_E$ for
the spaces of cadlag and continuous
paths starting from $E$, as in Section~\ref{LFC}. Write $e_k=(s_k,x_k)$
and recall the coordinate processes
$Z^k$ and their filtration $(\cZ_t)_{t\in\R}$, defined in
Section~\ref{LFC}. Define on $C_E$ the collision times
\[
\bar T^{jk}=\inf\bigl\{t\ge s_j\vee s_k\dvtx
Z_t^j=Z_t^k\bigr\}.
\]
The law $\bar\mu_E$ on $C_E$ of coalescing Brownian motions \emph{on the
line} then has the following martingale
characterization: for all $j,k$, the processes
$(Z^k_t)_{t\ge s_k}$ and $(Z^j_tZ^k_t-(t-\bar T^{jk})^+)_{t\ge s_j\vee
s_k}$ are both continuous local martingales in the filtration $(\cZ
_t)_{t\in\R}$.

For small $\ve$, we shall need to quantify the localization of a
disturbance in terms of the smallest constant
$\l=\l(f,\ve)\in(0,1]$ such that
\[
\label{lambdadefT} \rho\int_0^1\bigl|\tilde f(x+a)
\tilde f(x)\bigr|\,dx\le\l,\qquad  a\in[\ve\l,1-\ve\l].
\]

%
\begin{proposition}\label{LAFT}
The joint distribution $\mu_E^{f,\ve}$ of the family of trajectories
$(X^{e,\ve}\dvtx e\in E)$ in the $\ve$-scale Poisson disturbance flow with
disturbance $f$
converges weakly to the coalescing Brownian law $\bar\mu_E$ on $D_E$,
uniformly in $f\in\cD^*$, as $\ve\to0$ with $\ve^3\rho(f)\to
\infty$ and
$\l(f,\ve)\to0$.
\end{proposition}

\begin{pf}
Write $X^k$ for $X^{e_k,\ve}$ within the proof.
The family of laws $\{\mu^{f,\ve}_E\dvtx f\in\cD^*,\ve\in(0,1]\}$
is tight
on $D_E$.
Let $\mu$ be a weak limit law of this family for the limit $\ve\to0$
with $\ve^3\rho(f)\to\infty$ and $\l=\l(f,\ve)\to0$.
Then, as in Proposition~\ref{LAF}, under $\mu$, for all $j$, the
processes $(Z^j_t\dvtx t\ge s_j)$ and $((Z^j_t)^2-t\dvtx t\ge s_j)$ are
continuous local martingales.
For all $j,k$, the process
\[
X^j_tX^k_t-\int
_{s_j\vee s_k}^tb\bigl(\ve X^j_s,
\ve X^k_s\bigr)\,ds,\qquad  t\ge s_j\vee
s_k,
\]
is a martingale. Note that $|b(\ve X^j_s,\ve X^k_s)|\le\l$ until
$|X_t^j-X_t^k|$ leaves $[\l,\ve^{-1}-\l]$.
Define for $R\ge1$
\[
\bar T^{jk,R}=\inf\bigl\{t\ge s_j\vee s_k\dvtx
\bigl|Z_t^j-Z_t^k\bigr|\notin[1/R,R]\bigr\}
\]
then, $\bar T^{jk,R}\ua\bar T^{jk}$ everywhere on $C_E$ as $R\to
\infty$.
Under $\mu$, the process $(Z^j_tZ^k_t\dvtx\break   s_j\vee s_k\le t<\bar T^{jk,R})$
is a local martingale for all $R$,
so $(Z^j_tZ^k_t\dvtx s_j\vee s_k\le t<\bar T^{jk})$ is also a local martingale.
Now $\mu$ inherits from the laws $\mu^{f,\ve}_E$ the property
that, almost surely, the process $(Z^j_t-Z^k_t\dvtx t\ge s_j\vee s_k)$ does
not change sign.
Hence, $Z^j_t-Z^k_t$ is constant for $t\ge\bar T^{jk}$.
It follows that $(Z^j_tZ^k_t-(t-\bar T^{jk})^+)_{t\ge s_j\vee s_k}$ is
a continuous local martingale.
Hence, $\mu=\bar\mu_E$.
\end{pf}

We obtain state-spaces for flows on the line by replacing $\cD$ by
$\bar
\cD$ in the definitions made in Sections~\ref{CBF} and \ref{SKOR},
and replacing the metric $d_\cD$ by
%
%
\begin{equation}
\label{ddbardef} d_{\bar\cD}(f,g)=\sum_{n=1}^\infty2^{-n}
\sup_{t\in[-n,n]} \bigl(\bigl|f^\times(t)-g^\times(t)\bigr|
\wedge1 \bigr).
\end{equation}
Denote by $C^\circ(\R,\bar\cD)$ the set of continuous weak flows with
values in $\bar\cD$.
Define the coordinate processes $Z^e=Z^{e,+}$ and $Z^{e,-}$ and their
filtration $(\cF_t)_{t\in\R}$ on $C^\circ(\R,\bar\cD)$
just as for $C^\circ(\R,\cD)$ in Section~\ref{CBF}.
The collision time $\bar T^{ee'}\dvtx C^\circ(\R,\bar\cD)\to
[0,\infty]$, for
$e=(s,x)$ and $e'=(s',x')$, is now given by
\[
\bar T^{ee'}(\phi)=\inf\bigl\{t\ge s\vee s'\dvtx
Z^e_t(\phi)=Z^{e'}_t(\phi)\bigr\}.
\]
The following result is proved in \cite{TE}, Section~9 and, analogously
to Theorem~\ref{UBPR}, shows that there exists a unique probability
measure on $C^\circ(\R,\bar\cD)$ under which the left and right
coordinate processes $Z^{e,\pm}$ agree almost surely for all $e \in\R
^2$ and behave as coalescing Brownian motions.

%
\begin{theorem}\label{TUBPR}
There exists a unique Borel probability measure $\bar\mu_A$ on
$C^\circ
(\R,\bar\cD)$ under which,
for all $e,e'\in\R^2$, the processes $(Z^e_t)_{t\ge s(e)}$ and
$(Z^e_tZ_t^{e'}-(t-\bar T^{ee'})^+)_{t\ge s(e)\vee s(e')}$
are continuous local martingales for $(\cF_t)_{t\in\R}$.
Moreover, for all $e\in\R^2$, we have $Z^{e,+}=Z^{e,-}$ $\bar\mu
_A$-almost surely.
\end{theorem}

We call any $C^\circ(\R,\bar\cD)$-valued random variable with law
$\bar
\mu_A$ a \emph{coalescing Brownian flow}.
The space $D^\circ(\R,\bar\cD)$ of cadlag weak flows $(\phi_I\dvtx
I\sse\R)$ with
$\phi_I\in\bar\cD$ for all $I$ is defined analogously to $D^\circ
(\R,\cD)$.
The Skorokhod-type metric on $D^\circ(\R,\bar\cD)$ is defined just as
for $D^\circ(\R,\cD)$,
except that the metric of the uniform norm on $\cS$ is replaced by a
metric of uniform convergence on compacts on the space $\bar\cS$ of
contractions on $\R$. The following result follows from \cite{TE}, Lemma~14.1.
It extends \cite{FN}, Theorem~4.1, in allowing
processes with jumps in time.
Note that the additional noncrossing criterion needed in \cite{FN}
holds automatically in the space of weak flows.

%
\begin{theorem}\label{WCCT}
Let $(\mu_n\dvtx n\in\N)$ and $\mu$ be Borel probability measures
on $D^\circ
(\R,\bar\cD)$.
Assume that $Z^{E,-}=Z^{E,+}$ holds $\mu_n$-almost surely for all $n$
and $\mu$-almost surely.
Assume further that, for any finite sequence $E$ in $\R^2$, we have
$\mu
_n\circ(Z^E)^{-1}\to\mu\circ(Z^E)^{-1}$
weakly on $D_E$. Then $\mu_n\to\mu$ weakly on $D^\circ(\R,\bar\cD)$.
\end{theorem}

The $\ve$-scale Poisson disturbance flow $\Phi^\ve$ with disturbance $f$
may be considered as a Borel random variable in $D^\circ(\R,\bar\cD)$.
Moreover, for all $e\in\R^2$, we have $Z^{e,-}(\Phi^\ve
)=Z^{e,+}(\Phi
^\ve)$ almost surely.
The same is true in the lattice case.
Hence, Proposition~\ref{LAFT} and Theorems \ref{TUBPR} and \ref{WCCT}
imply the following local limit theorem.

%
\begin{theorem}\label{MAINT}
The $\ve$-scale Poisson disturbance flow with disturbance $f$
and the $\ve$-scale lattice disturbance flow with disturbance $f$
both converge weakly to the coalescing Brownian flow on the line
on $D^\circ(\R,\bar\cD)$,
uniformly in $f\in\cD^*$, as $\ve\to0$ with $\ve^3\rho(f)\to
\infty$ and
$\l(f,\ve)\to0$.
\end{theorem}

\begin{appendix}
\section*{Appendix}\label{app}
\subsection{Some properties of the space $\cD$ of nondecreasing functions
of degree~$1$}
We give proofs in this subsection of a number of assertions made in
Section~\ref{CBF}.

%
\begin{propositionn}
\label{fcrossbijection}
The map $f\mapsto f^\times\dvtx\cD\to\cS$ is a well-defined
bijection, with
inverse given by
\begin{eqnarray*}
f^-(x)&=&\inf\bigl\{t+f^\times(t)\dvtx t\in\R,x=t-f^\times(t)
\bigr\},
\\
f^+(x)&=&\sup\bigl\{t+f^\times(t)\dvtx t\in\R,x=t-f^\times(t)
\bigr\}.
\end{eqnarray*}
\end{propositionn}

\begin{pf}
Recall that $f^\times(t)=t-x$, where $x$ is the unique point such that
$f^-(x)\le2t-x\le f^+(x)$.
The periodicity of $f^\times$ is an easy consequence of the degree $1$
condition. We now show that
$f^\times$ is a contraction. Fix $s,t\in\R$ and suppose that
$f^\times(s)=s-y$. Switching the roles of $s$ and $t$ if necessary, we
may assume without
loss that $x\ge y$. If $x=y$, then $f^\times(s)-f^\times(t)=s-t$. On
the other hand, if $x>y$, then
$2s-y\le f^+(y)\le f^-(x)\le2t-x$, so
\begin{eqnarray*}
-(t-s)&\le&-(t-s)+(2t-x)-(2s-y)=f^\times(t)-f^\times(s)\\
&=&(t-s)-(x-y)<t-s.
\end{eqnarray*}
In both cases, we see that $|f^\times(t)-f^\times(s)|\le|t-s|$. Hence,
$f^\times\in\cS$.

Suppose now that $g\in\cS$. Consider, for each $x\in\R$, the set
\[
I_x=\bigl\{t+g(t)\dvtx t\in\R,x=t-g(t)\bigr\}.
\]
Since $g$ is a contraction, these sets are all intervals,
and, since $g$ is bounded, they cover $\R$.
For $x,y\in\R$ with $x>y$, and for $s,t\in\R$ with $x=t-g(t),y=s-g(s)$,
we have
$t-s-(g(t)-g(s))=x-y>0$, so $s\le t$, and so
\[
t+g(t)-\bigl(s+g(s)\bigr)=t-s+\bigl(g(t)-g(s)\bigr)\ge0.
\]
Define $h^+(y)=\sup I_y$ and $h^-(x)=\inf I_x$.
We have shown that $h^+(y)\le h^-(x)$.
Moreover, since the intervals $I_x$ cover $\R$, the functions $h^\pm$
must be
the left-continuous and right-continuous versions of a nondecreasing
function $h$,
which then has the degree $1$ property, because $g$ is periodic.
Thus, $h\in\cD$.

For each $t\in\R$, we have $h^\times(t)=t-x$,
where $2t-x\in I_x$, and so $2t-x=s+g(s)$ for some $s\in\R$
with $x=s+g(s)$. Then $s=t$ and so $h^\times(t)=g(t)$. Hence,
$h^\times=g$.
On the other hand, if we take $g=f^\times$ and if $x$ is a point of
continuity of $f$,
then we find $I_x=\{f(x)\}$, so $h^+(x)=h^-(x)=f(x)$. Hence, $h=f$.
We have now shown that $f\mapsto f^\times\dvtx\cD\to\cS$ is a bijection,
and that its inverse
has the claimed form.
\end{pf}

%
\begin{propositionn}
For $f,g\in\cD$ and $\ve>0$,
\begin{eqnarray}
d_\cD(f,g)\le\ve\quad\iff\quad f^-(x-\ve)-\ve\le g^-(x)\le g^+(x) \le f^+(x+
\ve)+\ve
\nonumber\\
\eqntext{\mbox{for all }x\in\R.}
\end{eqnarray}
Moreover, for any sequence $(f_n\dvtx n\in\N)$ in $\cD$,
\begin{eqnarray}
f_n\to f\mbox{ in }\cD\quad\iff\quad f_n^+(x)\to f(x)
\nonumber\\
\eqntext{\mbox{at all points $x\in\R$ where $f$ is continuous}.}
\end{eqnarray}
\end{propositionn}

\begin{pf}
Suppose that $d_\cD(f,g)\le\ve$ and that $x$ is a continuity point
of $g$.
Then $g(x)=t+g^\times(t)$ for some $t\in\R$ with $x=t-g^\times(t)$.
We must have $x+\ve\ge t-f^\times(t)$ and $g(x)\le t+f^\times(t)+\ve$,
so $f^+(x+\ve)+\ve\ge t+f^\times(t)+\ve\ge g(x)$.
Similarly $f^-(x-\ve)-\ve\le g(x)$. These inequalities extend to all
$x\in\R$
by taking left and right limits along continuity points.

Conversely, suppose that $t \in\R$ is such that $|f^\times
(t)-g^\times
(t)| = d_\cD(f,g)$
and let $x = t-g^\times(t)$ and $y=t-f^\times(t)$. Then $x$ is the
unique point with
$g^-(x)+x \leq2t \leq g^+(x)+x$ and $y$ is the unique point such that
$f^-(y)+y \leq2t \leq f^+(y)+y$. Hence, $f^-(x-\ve)-\ve\le g^-(x)\le
g^+(x) \le f^+(x+\ve)+\ve$
implies $y \in[x - \ve, x+ \ve]$ and so $d_\cD(f,g) = |y-x| \le\ve$.

It follows directly that for any sequence $(f_n\dvtx n\in\N)$ in $\cD$,
if $d_\cD(f_n,f) \rightarrow0$ as $n \rightarrow\infty$,
then $f_n^+(x)\to f(x)$ at all points $x\in\R$ where $f$ is continuous.

Now suppose $f_n^+(x)\to f(x)$ at all points $x\in\R$ where $f$ is continuous.
By equicontinuity, it will suffice to show that $f_n^\times(t)
\rightarrow f^\times(t)$ for each $t\in\R$.
Set $x=t-f^\times(t)$ and $x_n=t-f^\times_n(t)$.
Given $\ve>0$, choose $y_1\in(x-\ve,x)$ and $y_2\in(x,x+\ve)$,
both points
of continuity of $f$. Now $f(y_1)+y_1<2t<f(y_2)+y_2$, so there exists
$N\in\N$
such that for all $n \geq N$, we have
$f^+_n(y_1)+y_1<2t<f^+_n(y_2)+y_2$, which implies
$x_n\in[y_1,y_2]$, and hence $|f_n^\times(t) - f^\times(t)| < \ve$,
as required.
\end{pf}

%
\begin{propositionn}\label{WFL}
Suppose $f_n\to f, g_n\to g, h_n\to h$ in $\cD$ with
$h_n^+\le f_n^+\circ g_n^+$ for all $n$. Then $h^+\le f^+\circ g^+$.
\end{propositionn}

\begin{pf}
It will suffice to establish the inequality at all
points $x$ where $g$ and $h$ are both continuous.
Given $\ve>0$, since $f^+$ is right-continuous, there exists a point $y>g(x)$
where $f$ is continuous and such that $f(y)<f^+(g(x))+\ve$.
Then $f_n^+(y) < f^+(g(x))+\ve$ and $g_n^+(x)\le y$ eventually, so
\[
h_n^+(x)\le f_n^+\bigl(g_n^+(x)\bigr)\le
f_n^+(y)<f^+\bigl(g(x)\bigr)+\ve
\]
eventually. Hence, $h^+(x)=\lim_{n\to\infty}h_n^+(x)\le f^+(g^+(x))$,
as required.
\end{pf}

\subsection{Some properties of the continuous flow-space \texorpdfstring{$C^\circ(\mathbb{R},{\cal D})$}{$C^{circ}(\mathbb{R},{\cal D})$} and 
cadlag flow-space \texorpdfstring{$D^\circ(\mathbb{R},{\cal D})$}{$D^{circ}(\mathbb{R},{\cal D})$}}
We give proofs in this subsection of a number of assertions made in
Sections~\ref{CBF} and \ref{SKOR}.

%
\begin{propositionn}
For $(s,x)\in\R^2$ and $\phi\in D^\circ(\R,\cD)$, the map
\[
t\mapsto\phi_{(s,t]}^+(x)\dvtx[s,\infty)\to\R
\]
is cadlag, and is moreover continuous whenever $\phi\in C^\circ(\R
,\cD)$.
\end{propositionn}

\begin{pf}
Given $t\ge s$ and $\ve>0$, we can choose ${\delta}>0$ so that for
all $u\in
(t,t+{\delta}]$, $d_\cD(\phi_{(t,u]},\id)<\ve/2$.
For such $u$ and for $x$ a point of continuity of $\phi_{(s,t]}$, we have
\begin{eqnarray*}
\phi_{(s,t]}^+(x)-\ve&=&\phi_{(s,t]}^-(x)-\ve
\\
&\leq& \phi_{(t,u]}^-\circ\phi_{(s,t]}^-(x)
\\
&\leq& \phi_{(s,u]}^-(x)
\\
&\leq& \phi_{(s,u]}^+(x)
\\
&\leq& \phi_{(t,u]}^+ \circ\phi_{(s,t]}^+(x)
\\
&\leq& \phi_{(s,t]}^+(x) + \ve,
\end{eqnarray*}
so $|\phi_{(s,u]}^+(x)-\phi_{(s,t]}^+(x)|\leq\ve$. The final estimate
extends to all $x$ by right-continuity.
Hence, the map is right continuous. A similar argument shows that, for
$u\in(s,t)$, we have
$|\phi_{(s,u]}^+(x)-\phi_{(s,t)}^+(x)|\to0$ as $u\to t$, so that the
map has a left limit at $t$
given by $\phi_{(s,t)}^+(x)$. Finally, if $\phi\in C^\circ(\R,\cD
)$, then
$\phi_{(s,t)}=\phi_{(s,t]}$, so the map is continuous.
\end{pf}

%
\begin{propositionn}\label{PHICONT}
For all $\phi\in C^\circ(\R,\cD)$, the map $(s,t)\mapsto\phi
_{ts}\dvtx\{
(s,t)\dvtx s\le t\}\to\cD$ is continuous.
Moreover, for all $\phi\in D^\circ(\R,\cD)$ and for any sequence of
bounded intervals $I_n\to I$, we have $\phi_{I_n}\to\phi_I$.
\end{propositionn}

\begin{pf}
The first assertion follows from the second: given $\phi\in C^\circ
(\R
,\cD)$ and sequences $s_n\to s$ and $t_n\to t$, then,
passing to a subsequence if necessary, we can assume that $(s_n,t_n]\to
I$ for some interval $I$ with $\inf I=s$ and $\sup I=t$.
Then, by the second assertion, we have $\phi_{t_ns_n}\to\phi_I=\phi
_{ts}$, as required.

So, let us fix $\phi\in D^\circ(\R,\cD)$ and a sequence of bounded
intervals $I_n\to I$.
By combining the cadlag and weak flow properties, we can show the
following variant of the cadlag property: for all $t\in\R$, we have
%
%
\begin{equation}
\label{WFP2} \phi_{[s,t)}\to\id \qquad\mbox{as $s\uparrow t$},\qquad
\phi_{(t,u]}\to\id \qquad\mbox{as $u\downarrow t$}.
\end{equation}
For each $n$, there exist two disjoint intervals $J_n$ and $J_n'$,
possibly empty,
such that $I\triangle I_n=J_n\cup J_n'$. For any such $J_n$ and $J_n'$,
using the weak flow property, we obtain
\[
d_\cD(\phi_I,\phi_{I_n})\le\|
\phi_{J_n}-\id\|+\|\phi_{J_n'}-\id\|.
\]
Set $s=\inf I$, $s_n=\inf I_n$, $t=\sup I$ and $t_n=\sup I_n$.
Then $s_n\to s$, $t_n\to t$, and
\begin{eqnarray*}
&&\mbox{if $s\in I$ then $s\in I_n$ eventually},\qquad \mbox{if $s
\notin I$ then $s\notin I_n$ eventually},\\
&&\mbox{if $t\in I$ then $t\in I_n$ eventually},\qquad \mbox{if $t
\notin I$ then $t\notin I_n$ eventually}.
\end{eqnarray*}
Hence, using the cadlag property or (\ref{WFP2}), or both, we find that
$\phi_{J_n}\to\id$ and $\phi_{J_n'}\to\id$, which proves the proposition.
\end{pf}

%
\begin{propositionn}
The metrics $d_C$ and $d_D$ generate the same topology on $C^\circ(\R
,\cD)$.
\end{propositionn}

\begin{pf}
On comparing the definitions of $d_C^{(n)}$ and $d_D^{(n)}$ for each
$n\in\N$, and considering the choice $\l=\id$,
we see that $d_D\le d_C$. Hence, it will suffice to show, given $\phi
\in C^\circ(\R,\cD)$, $n\in\N$ and $\ve>0$,
that there exists $\ve'>0$ such that, for all $\psi\in C^\circ(\R
,\cD
)$, we have $d_C^{(n)}(\phi,\psi)<\ve$
whenever $d^{(n+1)}_D(\phi,\psi)<\ve'$. By the preceding proposition,
there exists a ${\delta}\in(0,1]$ such that
$d_\cD(\phi_{ts},\phi_{t's'})<\ve/2$ whenever $|s-s'|,|t-t'|\le
{\delta}$ and
$s,t\in(-n,n)$. Set $\ve'={\delta}\wedge(\ve/2)$
and suppose that $d^{(n+1)}_D(\phi,\psi)<\ve'$. Then there exists an
increasing homeomorphism $\l$ of $\R$, with
$|\l(t)-t|\le{\delta}$ for all $t$, such that, for all intervals
$I$, we have
$\|\chi_{n+1}(I)\psi_I^\times-\chi_{n+1}(\l(I))\phi_{\l
(I)}^\times\|<\ve/2$.
Given $s,t\in(-n,n)$ with $s<t$, take $I=(s,t]$.
Then $\chi_{n+1}(I)=\chi_{n+1}(\l(I))=1$, so $d_\cD(\phi_{\l(t)\l
(s)},\psi_{ts})=\|\psi_I^\times-\phi_{\l(I)}^\times\|<\ve/2$.
But then, for all such $s,t$, we have
\[
d_\cD(\phi_{ts},\psi_{ts})\le
d_\cD(\phi_{ts},\phi_{\l(t)\l(s)})+d_\cD(
\phi_{\l(t)\l(s)},\psi_{ts})<\ve,
\]
so $d_C^{(n)}(\phi,\psi)<\ve$, as required.
\end{pf}

%
\begin{propositionn}
\label{compsep}
The metric spaces
$(C^\circ(\R,\cD),d_C)$ and $(D^\circ(\R,\cD),d_D)$ are complete
and separable.
\end{propositionn}

\begin{pf}
The argument for completeness is a variant of the corresponding
argument for the usual Skorokhod space $D(\R, S)$ of cadlag paths in
complete separable metric
space $S$, as found, for example, in \cite{B}.
Suppose then that $(\psi^n)_{n \geq1}$ is a Cauchy sequence in
$D^\circ
(\R,\cD)$.
There exists a subsequence $\phi^k=\psi^{n_k}$ such that
$d_D^{(n)}(\phi
^n, \phi^{n+1}) < 2^{-n}$ for all $n\ge1$.
It will suffice to find a limit in $D^\circ(\R,\cD)$ for $(\phi^n)_{n
\geq1}$.
Recall the definition of $\gamma$ from \eqref{gammadef}.
There exist increasing homeomorphisms $\k_n$ of $\R$
for which $\gamma(\k_n) < 2^{-n}$ and
\[
d_\cD\bigl(\phi^n_I,
\phi^{n+1}_{\k_n(I)}\bigr)<2^{-n},\qquad I\cup
\k_n(I)\sse(-n,n).
\]
For each $n\ge1$, the sequence $(\k_{n+m} \circ\cdots\circ\k
_{n})_{m\ge1}$ converges
uniformly on $\R$ to an increasing homeomorphism, $\lambda_n$ say,
with $\gamma(\lambda_n) < 2^{-n+1}$. Then $\k_n\circ\l_n^{-1}=\l
_{n+1}^{-1}$, so
\[
d_\cD\bigl(\phi^n_{\lambda_n^{-1}(I)},
\phi^{n+1}_{\lambda_{n+1}^{-1}(I)}\bigr) < 2^{-n},\qquad  I\sse(-n+1,n-1).
\]
So, for all $m\ge n$,
%
%
\begin{equation}
\label{PHMN} d_\cD\bigl(\phi^n_{\lambda_n^{-1}(I)},
\phi^{n+m}_{\lambda_{n+m}^{-1}(I)}\bigr) < 2^{-n+1},\qquad  I\sse(-n+1,n-1).
\end{equation}
Hence, for all bounded intervals $I\sse\R$,
$(\phi^n_{\lambda_n^{-1}(I)})_{n \geq1}$ is a
Cauchy sequence in $\cD$, which, since $\cD$ is complete, has a limit
$\phi_I\in\cD$.
On letting $m\to\infty$ in (\ref{PHMN}), we obtain
\[
d_\cD\bigl(\phi^n_{\lambda_n^{-1}(I)},\phi_I
\bigr) < 2^{-n+1}, \qquad I\sse(-n+1,n-1).
\]
By Proposition~\ref{WFL}, $\phi=(\phi_I\dvtx I\sse\R)$ has the
weak flow property.
To see that $\phi$ is cadlag, suppose given $\ve>0$ and $t\in\R$.
Choose $n$ such that
$2^{-n+1}\le\ve/3$ and $|t|\le n-2$. Then choose ${\delta}\in(0,1]$
such that
\[
d_\cD\bigl(\phi^n_{\l_n^{-1}(s,t)},\id\bigr)<\ve/3,\qquad
d_\cD\bigl(\phi^n_{\l
_n^{-1}(t,u)},\id\bigr)<\ve/3
\]
whenever $s\in(t-{\delta},t)$ and $u\in(t,t+{\delta})$. For such
$s$ and $u$, we
then have
\[
d_\cD(\phi_{(s,t)},\id)<\ve,\qquad d_\cD(
\phi_{(t,u)},\id)<\ve.
\]
Hence, $\phi\in D^\circ(\R,\cD)$. For $m\le n-3$, we have
\begin{eqnarray*}
d_D^{(m)}\bigl(\phi^n,\phi\bigr) &\le&\g(
\l_n)\vee\sup_{I\sse(-m-2,m+2)}\bigl\|\chi_m\bigl(
\l_n^{-1}(I)\bigr)\phi_{\l
_n^{-1}(I)}^{n\times}-
\chi_m(I)\phi_I^\times\bigr\|
\\
&\le&\g(\l_n)\vee\sup_{I\sse(-m-2,m+2)} \bigl
\{d_\cD\bigl(\phi^n_{\lambda
_n^{-1}(I)},\phi_I
\bigr)+\g(\l_n)\bigl\|\phi_I^\times\bigr\| \bigr\}
\\
&\le&2^{-n+1}\Bigl(1+\sup_{I\sse(-m-2,m+2)}\bigl\|\phi_I^\times
\bigr\|\Bigr).
\end{eqnarray*}
Hence, $d_D(\phi^n,\phi)\to0$ as $n\to\infty$.
We have shown that $D^\circ(\R,\cD)$ is complete.
If the sequence $(\phi^n)_{n \geq1}$ in fact lies in $C^\circ(\R
,\cD
)$, then
by an obvious variation of the argument for the cadlag property, the limit
$\phi$ also lies in $C^\circ(\R,\cD)$. Hence, $C^\circ(\R,\cD)$
is also
complete.
In particular, $C^\circ(\R,\cD)$ is a closed subspace in $D^\circ
(\R,\cD)$.

We turn to the question of separability.
Let us write $D_N$ for the set of those $\phi\in D^\circ(\R,\cD)$
such that:
\begin{longlist}[(ii)]
\item[(i)] for some $n\in\N$ and some rationals $t_1<\cdots<t_n$,
we have
$\phi_J=\id$ for all time intervals $J$, which do not intersect the set
$\{t_1,\ldots,t_n\}$;
\item[(ii)] for all other time intervals $I$, the maps $\phi_I$ and
$\phi_I^{-1}$ on $\R$
are constant on all space intervals which do not intersect $2^{-N}\Z$.
\end{longlist}
Note that each $\phi\in D_N$ is determined by the maps $\phi_{(t_k,t_m]}$,
for integers $0\le k<m\le n$, where $t_0<t_1$,
and for each of these maps there are only countably many possibilities
(finitely many
if we insist that $\phi(0)\in[0,1)$).
Hence, $D_N$ is countable and so is $D_*=\bigcup_{N\ge1}D_N$.
We shall show that $D_*$ is also dense in $D^\circ(\R,\cD)$.

Fix $\phi\in D^\circ(\R,\cD)$ and $n_0\ge1$. It will suffice to find,
for a given $\ve>0$,
a $\psi\in D_*$ with $d^{(n_0)}_D(\phi,\psi)<\ve$. By the cadlag
property and compactness,
there exist $n\in\N$ and reals $s_1<\cdots<s_n$ in $I_0=(-n_0-1,n_0+1)$
such that
$d_\cD(\phi_I,\id)<\ve/4$ for every subinterval $I$ of $I_0$, which
does not intersect $\{s_1,\ldots,s_n\}$.
To see this, let
\[
A=\bigl\{t \in I_0 \dvtx d_\cD(\phi_{(s,t]},
\phi_{(s,t)})\geq\ve/4 \mbox{ for some } s<t \bigr\}.
\]
If $A$ contains infinitely many points, then there exists a sequence
$(u_m)_{m \in\N}$ in $A$ and $u \in\R$ such that $u_m \to u$ strictly
monotonically. Suppose that $u_m \uparrow u$. Then, as in the proof of
Proposition~\ref{PHICONT}, $\|\phi_{(u_m,u)}-\id\| < \ve/8$ and $\|
\phi
_{[u_m,u)}-\id\| < \ve/8$ for $m$ is sufficiently large. But then, for
all $s<u_m$,
\begin{eqnarray*}
d_\cD(\phi_{(s,u_m]}, \phi_{(s,u_m)}) &\leq&
d_\cD(\phi_{(s,u_m]}, \phi_{(s,u)}) +
d_\cD(\phi_{(s,u)}, \phi_{(s,u_m)})
\\
&\leq& \| \phi_{(u_m,u)}-\id\| + \| \phi_{[u_m,u)}-\id\|
\\
&<& \ve/4,
\end{eqnarray*}
contradicting $u_m \in A$. A similar contradiction arises if $u_m
\downarrow u$, so $A$ contains finitely many points. Therefore, $I_0
\setminus A$ consists of the disjoint union of finitely many open
intervals. It remains to show that if $J$ is one of these intervals,
there exists some $\eta>0$ such that if an interval $I \subseteq J$ and
$\sup I - \inf I < \eta$, then $d_\cD(\phi_I, \id) < \ve/4$. If not
then, there exists a sequence of intervals $I_m \subseteq J$ with $\sup
I_m - \inf I_m < m^{-1}$ and $d_\cD(\phi_{I_m}, \id) \geq\ve/4$. By
restricting to a subsequence if necessary $I_m \to I$ where
$I=\varnothing
$ or $\{t\}$ for some $t \in J$. Therefore, $\phi_{I_m} \to\phi_I$.
But $\phi_{\varnothing}=\id$ and $d_\cD(\phi_{\{t\}}, \id) < \ve
/4$ for
all $t \notin A$, which contradicts $d_\cD(\phi_{I_m}, \id) \geq\ve/4$
for all $m$.

Next we can find rationals $t_1<\cdots<t_n$ in $I_0$ and
an increasing homeomorphism $\l$ of $\R$, with $\l(t)=t$ for
$t\notin
I_0$, with $\g(\l)\sup_{I\sse I_0}\|\phi^\times_I\|<\ve/4$,
and such that $\l(t_m)=s_m$ for all $m$.
Set $s_0=t_0=-n_0-1$.

For $f\in\cD$, write $\Delta(f)$ for the set of points where $f$ is not
continuous.
Define, for $m=0,1,\ldots,n$,
\[
\Delta_m=\bigcup_{k=0}^{m-1}
\Delta\bigl(\phi_{(s_k,s_m]}^{-1}\bigr)\cup\bigcup
_{k=m+1}^n\Delta(\phi_{(s_m,s_k]}).
\]
Then $\Delta_m$ is countable, so we can choose $N\ge1$ with $16 \cdot
2^{-N}\le\ve$ and choose
$\ve_m\in\R$ with $|\ve_m|\le2^{-N}$ such that
\[
\t_m(\Delta_m)\cap2^{-N}\Z=\es,\qquad m=0,1,
\ldots,n,
\]
where $\t_m(x)=x+\ve_m$. Set
\[
{\delta}^-(x)=2^N\bigl\lceil2^{-N}x\bigr\rceil,\qquad
{\delta}^+(x)=2^N\bigl\lfloor2^{-N}x\bigr\rfloor+1.
\]
Note that ${\delta}=\{{\delta}^-,{\delta}^+\}\in\cD$. Define for
$0\le k<m\le n$
\begin{eqnarray*}
\psi_{(t_k,t_m]}^-&=&\bigl({\delta}^{-1}\bigr)^-\circ(
\t_m)^{-1}\circ\phi_{(s_k,s_m]}^-\circ\t_k
\circ{\delta}^-,
\\
\psi_{(t_k,t_m]}^+&=&\bigl({\delta}^{-1}\bigr)^+\circ(
\t_m)^{-1}\circ\phi_{(s_k,s_m]}^+\circ\t_k
\circ{\delta}^+.
\end{eqnarray*}
Then $\psi_{(t_k,t_m]}=\{\psi_{(t_k,t_m]}^-,\psi_{(t_k,t_m]}^+\}\in
\cD$
by our choice of $\ve_k$ and $\ve_m$.
Moreover, ${\delta}^+\circ({\delta}^{-1})^+\ge\id$ and ${\delta
}^-\circ({\delta}^{-1})^-\le\id$ so,
for $0\le m<m'<m''\le n$, we obtain the inequalities
\[
\psi_{(t_{m'},t_{m''}]}^-\circ\psi_{(t_m,t_{m'}]}^-\le\psi
_{(t_m,t_{m''}]}^- \le
\psi_{(t_m,t_{m''}]}^+\le\psi_{(t_{m'},t_{m''}]}^+\circ\psi
_{(t_m,t_{m'}]}^+
\]
from the corresponding inequalities for $\phi$.
We use the equations $\|{\delta}-\id\|=2^{-N}$ and $\|\t_m-\id\|
=|\ve_m|$
to see that
\[
d_\cD(\phi_{(s_k,s_m]},\psi_{(t_k,t_m]})\le4
\cdot2^{-N},\qquad  0\le k<m\le n.
\]
For all intervals $J$ such that
$J\cap\{t_1,\ldots,t_n\}=\{t_{k+1},\ldots,t_m\}$, define $\psi
_J=\psi
_{(t_k,t_m]}$.
For such intervals $J$, with $J\sse I_0$, we have
$d_\cD(\phi_{(s_k,s_m]\sm\l(J)},\id)<\ve/4$ and $d_\cD(\phi_{\l
(J)\sm
(s_k,s_m]},\id)<\ve/4$; so,
using the weak flow property for $\phi$,
\begin{eqnarray*}
d_\cD(\psi_J, \phi_{\l(J)}) &\le&
d_\cD(\psi_{(t_k,t_m]},\phi_{(s_k,s_m]}) +d_\cD(
\phi_{(s_k,s_m]},\phi_{\l(J)})
\\
&\le& 4 \cdot2^{-N}+2\ve/4
\\
&<&3\ve/4.
\end{eqnarray*}
Define $\psi_J=\id$ for all intervals $J$ which do not intersect $\{
t_1,\ldots,t_n\}$.
For such intervals $J$ with $J\sse I_0$, we have $d_\cD(\psi_J, \phi
_{\l
(J)})\le d_\cD(\id,\phi_{\l(J)})\le\ve/4$.
Now $\psi\in D_N$ and
\[
d^{(n_0)}_D(\phi,\psi)\le\g(\l)\vee\sup
_{J\sse I_0}\bigl\{d_\cD(\psi_J,\phi
_{\l(J)})+\g(\l)\bigl\|\phi_J^\times\bigr\|\bigr\}<\ve
\]
as required.
This proves that $D^\circ(\R,\cD)$ is separable and, since $C^\circ
(\R
,\cD)$
is a closed subspace of $D^\circ(\R,\cD)$, it follows that $C^\circ
(\R
,\cD)$
is also separable.
\end{pf}

%
\begin{propositionn}
For all $s,t\in\R$ with $s<t$, and all $x\in\R$, the map
$\phi\mapsto\phi_{ts}^+(x)$ on $C^\circ(\R,\cD)$ is Borel measurable.
Moreover, the Borel $\s$-algebra on $C^\circ(\R,\cD)$ is generated by
the set of all such maps with $s,t$ and $x$ rational.

For all bounded intervals $I\sse\R$ and all $x\in\R$, the
map $\phi\mapsto\phi_I^+(x)$ on $D^\circ(\R,\cD)$ is Borel
measurable. Moreover, the Borel $\s$-algebra on $D^\circ(\R,\cD)$ is
generated by
the set of all such maps with $I=(s,t]$ and with $s,t$ and $x$ rational.
\end{propositionn}

\begin{pf}
The assertions for $C^\circ(\R,\cD)$ can be proved more simply than
those for $D^\circ(\R,\cD)$.
We omit details of the former, but note that these follow also from the
latter, by general
measure theoretic arguments, given what we already know about the two spaces.

The proof for $D^\circ(\R,\cD)$ is an adaptation of the analogous
result for the classical Skorokhod
space; see, for example, \cite{MR1943877}, page 335.
We prove first the Borel measurability of the evaluation maps.
Given a bounded interval $I$ and $x\in\R$, we can find $s_n,t_n\in\R$
such that $(s_n,t_n]\to I$ as $n\to\infty$.
Then $\phi^+_I(x)=\lim_{m\to\infty}\lim_{n\to\infty}\phi
^+_{(s_n,t_n]}(x+1/m)$, by Proposition~\ref{PHICONT}.
Hence, it will suffice to consider intervals $I$ of the form $(s,t]$.
Fix $s,t$ and $x$ and define for each $m,n\in\N$
a function $F_{m,n}$ on $D^\circ(\R,\cD)$ by
\[
F_{m,n}(\phi)=\int_s^{s+1/n}\int
_t^{t+1/n}\int_x^{x+1/m}
\phi^+_{(s',t']}\bigl(x'\bigr)\,dx'\,dt'\,ds'.
\]
Suppose $\phi^k\to\phi$ in $D^\circ(\R,\cD)$. We can choose increasing
homeomorphisms $\l_k$ of $\R$ such that,
$\g(\l_k)\to0$ and, uniformly in $r\in[s-1,s+1]$ and $u\in[t-1,t+1]$,
we have
\[
d_\cD\bigl(\phi^k_{\l_k(r,u]},\phi_{(r,u]}
\bigr)\to0.
\]
Define
\[
f(r,u)=\int_x^{x+1/m}\phi_{\l(r,u]}
\bigl(x'\bigr)\,dx',\qquad f_k(r,u)=\int
_x^{x+1/m}\phi^k_{\l_k(r,u]}
\bigl(x'\bigr)\,dx'.
\]
Then $f_k(r,u)\to f(r,u)$, uniformly in $r\in[s-1,s+1]$ and $u\in[t-1,t+1]$.
Set $\mu_k=\l_k^{-1}$. Then
\begin{eqnarray*}
F_{m,n}\bigl(\phi^k\bigr)&=&\int_{\mu_k(s)}^{\mu_k(s+1/n)}
\int_{\mu_k(t)}^{\mu
_k(t+1/n)}f_k(r,u)\,d
\l_k(u)\,d\l_k(r)
\\
&\to&\int_s^{s+1/n}\int_t^{t+1/n}f(r,u)\,du\,dr=
F_{m,n}(\phi),
\end{eqnarray*}
so $F_{m,n}$ is continuous on $D^\circ(\R,\cD)$. By Proposition~\ref
{PHICONT}, we have
\[
\phi_{(s,t]}^+(x)=\lim_{m\to\infty}\lim_{n\to\infty}
\frac{1}{mn^2}F_{m,n}(\phi).
\]
Hence, $\phi\mapsto\phi_{(s,t]}^+(x)$ is Borel measurable, as required.

Write now $\cE$ for the $\s$-algebra on $D^\circ(\R,\cD)$
generated by
all maps of this form
with $s,t$ and $x$ rational.
It remains to show that $\cE$ contains the Borel $\s$-algebra of
$D^\circ(\R,\cD)$.
Write $\{(I_k,z_k)\dvtx k\in\N\}$ for an enumeration of the set $\{
(s,t]\dvtx s,t\in\Q,s<t\}\times\Q$.
It is straightforward to show that, for all $k$, the map $\phi\mapsto
\phi_{I_k}^\times(z_k)$ is $\cE$-measurable.
Fix $n\in\N$, $\phi^0\in D^\circ(\R,\cD)$, $r\in(0,\infty)$ and
$k\in\N
$, and consider the set
\[
A(k,r)=\bigl\{\phi\in D^\circ(\R,\cD)\dvtx\bigl(\chi_n(I_1)
\phi_{I_1}^\times(z_1),\ldots,\chi_n(I_k)
\phi_{I_k}^\times(z_k)\bigr)\in B(k,r)\bigr\},
\]
where
\[
B(k,r)=\bigcup_\l\Bigl\{(y_1,
\ldots,y_k)\in\R^k\dvtx\max_{j\le k}\bigl|y_j-
\chi_n\bigl(\l(I_j)\bigr)\phi_{\l(I_j)}^{0\times}(z_j)\bigr|<r
\Bigr\},
\]
where the union is taken over all increasing homeomorphisms $\l$ of
$\R
$ with $\g(\l)<r$.
Note that $B(k,r)$ is an open set in $\R^k$, so $A(k,r)\in\cE$, so
$A=\bigcup_{m\in\N}\bigcap_{k\in\N}A(k,r-1/m)\in\cE$.

Consider the set
\[
C=\bigl\{\phi\in D^\circ(\R,\cD)\dvtx d_D^{(n)}
\bigl(\phi,\phi^0\bigr)<r\bigr\}.
\]
It is straightforward to check from the definition of $d_D^{(n)}$, that
$C\sse A$. Suppose that $\phi\in A$. We shall
show that $\phi\in C$. Then $C=A$, so $C\in\cE$, and since sets of this
form generate the Borel $\s$-algebra, we are done.

We can find an $m\in\N$ and, for each $k\in\N$, a $\l_k$ with $\g
(\l
_k)<r-1/m$ such that
\[
\bigl|\chi_n(I_j)\phi_{I_j}^\times(z_j)-
\chi_n\bigl(\l_k(I_j)\bigr)
\phi_{\l
_k(I_j)}^{0\times}(z_j)\bigr|<r-1/m,\qquad j=1,\ldots,k.
\]
Without loss of generality, we may assume that the sequence $(\l
_k\dvtx k\in
\N)$ converges uniformly on compacts, and
that its limit, $\l$ say, satisfies $\g(\l)\le r-1/m$.
By Proposition~\ref{PHICONT}, for each $j$, there is an interval $\hat
I_j$, having the same endpoints as $I_j$
such that $\phi_{\l(\hat I_j)}$ is a limit point in $\cD$ of the
sequence $(\phi_{\l_k(I_j)}\dvtx k\in\N)$,
so $\phi_{\l(\hat I_j)}^\times$ is a limit point in $\cS$ of the
sequence $(\phi_{\l_k(I_j)}^\times\dvtx k\in\N)$. Then
\[
\bigl|\chi_n(I_j)\phi_{I_j}^\times(z_j)-
\chi_n\bigl(\l(\hat I_j)\bigr)\phi_{\l(\hat
I_j)}^{0\times}(z_j)\bigr|
\le r-1/m
\]
for all $j$. For all bounded intervals $I$ and all $z\in\R$, we can
find a sequence $(j_p\dvtx p\in\N)$ such that
$I_{j_p}\to I$, $\hat I_{j_p}\to I$ and $z_{j_p}\to z$. So, we obtain
\[
\bigl|\chi_n(I)\phi_{I}^\times(z)-
\chi_n\bigl(\l(I)\bigr)\phi_{\l(I)}^{0\times}(z)\bigr|\le
r-1/m.
\]
Hence, $d_D^{(n)}(\phi,\phi^0)\le r-1/m$ and $\phi\in C$, as we claimed.
\end{pf}

Recall that, for $e=(s,x)\in\R^2$ and $\phi\in D^\circ(\R,\cD)$,
we set
\[
Z^{e,\pm}(\phi)=\bigl(\phi^\pm_{(s,t]}(x)\dvtx t\ge s
\bigr)
\]
and for sequences $E=(e_k\dvtx k\in\N)$ in $\R^2$, we set $Z^{E,\pm
}=(Z^{e_k,\pm}\dvtx k\in\N)$.
Also
\[
C^{\circ,\pm}_E=\bigl\{Z^{E,\pm}(\phi)\dvtx\phi\in
C^\circ(\R,\cD)\bigr\},\qquad D^{\circ,\pm}_E=\bigl
\{Z^{E,\pm}(\phi)\dvtx\phi\in D^\circ(\R,\cD)\bigr\}
\]
and
\begin{eqnarray*}
D^\circ(E)&=&\bigl\{\phi\in D^\circ(\R,\cD)\dvtx
Z^{E,+}(\phi)=Z^{E,-}(\phi)\bigr\},
\\
D^\circ_E&=&\bigl\{Z^E(\phi)\dvtx\phi\in
D^\circ(E)\bigr\}.
\end{eqnarray*}

%
\begin{propositionn}
Let $E$ be a countable subset of $\R^2$ containing\footnote{The role of
$\Q$ here could be played by any countable dense subset of $\R$. The
same comment applies to Propositions \ref{MPI} and \ref{PCP}.} $\Q^2$.
Then $Z^{E,+}\dvtx C^\circ(\R,\cD)\to C^{\circ,+}_E$ is a bijection,
$C^{\circ,+}_E$ is a measurable subset of $C_E$, and
the inverse bijection $\Phi^{E,+}\dvtx C^{\circ,+}_E\to C^\circ(\R
,\cD)$ is
a measurable map.
Moreover, $Z^{E,+}\dvtx D^\circ(\R,\cD)\to D^{\circ,+}_E$ is also a
bijection,
$D^{\circ,+}_E$ is a measurable subset of $D_E$
and the inverse bijection $\Phi^{E,+}\dvtx D^{\circ,+}_E\to D^\circ
(\R,\cD)$
is also a measurable map.
Moreover, the same statements hold with $+$ replaced by $-$, we have
$D^\circ_E=D^{\circ,+}_E\cap D^{\circ,-}_E$
and $\Phi^{E,+}=\Phi^{E,-}$ on $D_E^\circ$.
\end{propositionn}

\begin{pf}
We discuss only the cadlag case.
The same comments apply as in the preceding proof about the
relationship of the cadlag and continuous cases.
It is straightforward to see from the density of $E$ in $\R^2$ and the
continuity properties of cadlag weak flows
that $Z^{E,+}$ and $Z^{E,-}$ are both injective on $D^\circ(\R,\cD)$.
We shall instead give an explicit description of the ranges $D_E^{\circ
,\pm}$ and explicit constructions of inverse maps
$\Phi^{E,+}$ and $\Phi^{E,-}$, which agree on $D_E^{\circ,+}\cap
D_E^{\circ,-}$, allowing us to establish measurability (as well as
injectivity).
Consider for $z\in D_E$ the conditions
%
%
\begin{equation}
\label{ZDEG1} z_t^{(s,x+n)}=z_t^{(s,x)}+n,\qquad
s,t,x\in\Q, s<t, n\in\Z
\end{equation}
and
%
%
\begin{equation}
\label{ZRIGHT} z_t^{(s,x)}=\inf_{y\in\Q,y>x}z_t^{(s,y)},\qquad
(s,x)\in E, t\in\Q, t>s.
\end{equation}
Under these conditions, define for $s,t\in\Q$ with $s<t$ and for
$x\in\R$,
\[
\Phi_{(s,t]}^-(x)=\sup_{y\in\Q,y<x}z_t^{(s,y)},\qquad
\Phi_{(s,t]}^+(x)=\inf_{y\in\Q,y>x}z_t^{(s,y)}.
\]
Then $\Phi_{(s,t]}=\{\Phi_{(s,t]}^-,\Phi_{(s,t]}^+\}\in\cD$ and
\[
\Phi_{(s,t]}^+(x)=z_t^{(s,x)}, \qquad s,t,x\in\Q, s<t.
\]
Now consider the following additional conditions on $z$:
%
%
\begin{eqnarray}
\label{ZFLOW} \Phi_{(t,u]}^-\circ\Phi_{(s,t]}^-\le
\Phi_{(s,u]}^-\le\Phi_{(s,u]}^+\le\Phi_{(t,u]}^+\circ
\Phi_{(s,t]}^+,
\nonumber
\\[-8pt]
\\[-8pt]
 \eqntext{s,t,u\in\Q, s<t<u}
\end{eqnarray}
and
%
for all $\ve>0$ and all $n\in\N$, there exist ${\delta}>0$,
$m\in\Z^+$
and $u_1,\ldots,u_m\in(-n,n)$ such that
%
%
\begin{equation}
\label{ZCADLAG} \|\Phi_{(s,t]}-\id\|<\ve
\end{equation}
%
whenever $s,t\in\Q\cap(-n,n)$ with $0<t-s<{\delta}$ and
$(s,t]\cap\{
u_1,\ldots,u_m\}=\es$.
%

Note that the inequalities between functions required in (\ref{ZFLOW})
hold whenever the same inequalities
hold between their restrictions to $\Q$, by left and\vadjust{\goodbreak} right continuity.
Note also that condition (\ref{ZCADLAG}) is equivalent to the following
condition involving quantifiers only over countable sets:
\begin{itemize}
\item[]for all rationals $\ve>0$ and all $n\in\N$, there exist a
rational ${\delta}>0$ and an $m\in\Z^+$ such that, for all rationals
$\eta>0$,
there exist rationals $s_1,t_1,\ldots,s_m,t_m\in(-n,n)$, with $s_i<t_i$
for all $i$ and with
$\sum_{i=1}^m(t_i-s_i)<\eta$, such that
%
\[
\|\Phi_{(s,t]}-\id\|<\ve
\]
%
%
whenever $s,t\in\Q\cap(-n,n)$ with $0<t-s<{\delta}$ and
$(s,t]\cap
((s_1,t_1]\cup\cdots\cup(s_m,t_m])=\es$.
\end{itemize}
Denote by $D_E^{*,+}$ the set of those $z\in D_E$ where conditions
(\ref
{ZDEG1}), (\ref{ZRIGHT}), (\ref{ZFLOW}) and (\ref{ZCADLAG})
all hold. Then $D^{*,+}_E$ is a measurable subset of $D_E$. Fix $z\in
D^{*,+}_E$. Given a bounded interval $I$, we can find
sequences of rationals $s_n$ and $t_n$ such that $(s_n,t_n]\to I$ as
$n\to\infty$. Then, by conditions (\ref{ZFLOW}) and (\ref{ZCADLAG}),
\[
d_\cD(\Phi_{(s_n,t_n]},\Phi_{(s_m,t_m]})\le\|
\Phi_{(s_n,s_m]}-\id\|+\| \Phi_{(t_n,t_m]}-\id\|\to0
\]
as $n,m\to\infty$. So the sequence $\Phi_{(s_n,t_n]}$ converges in
$\cD
$, with limit $\Phi_I$, say, and $\Phi_I$ does not
depend on the approximating sequences of rationals.
In the case where $I=I_1\oplus I_2$, there exists another sequence of
rationals $u_n$ such that $(s_n,u_n]\to I_1$
and $(u_n,t_n]\to I_2$ as $n\to\infty$. Hence, $\Phi=(\Phi_I\dvtx
I\sse\R)$
has the weak flow property, by Proposition~\ref{WFL}. It is
straightforward to deduce from (\ref{ZCADLAG}) that
$\Phi$ is moreover cadlag, so $\Phi=\Phi(z)\in D^\circ(\R,\cD)$.
It follows from its construction and the preceding proposition that the
map $z\mapsto\Phi(z)\dvtx D^{*,+}_E\to D^\circ(\R,\cD)$ is measurable.

Now, for all $z\in D^{*,+}_E$, we have $Z^{E,+}(\Phi(z))=z$ and for all
$\phi\in D^\circ(\R,\cD)$, we have $Z^{E,+}(\phi)\in D^{*,+}_E$ and
$\Phi(Z^{E,+}(\phi))=\phi$. Hence, $D_E^{\circ,+}=D^{*,+}_E$ and
$Z^{E,+}\dvtx D^\circ(\R,\cD)\to D_E^{\circ,+}$ is a bijection with
inverse $\Phi^{E,+}=\Phi$.

Consider now for $z\in D_E$ the condition
%
%
\begin{equation}
\label{ZLEFT} z_t^{(s,x)}=\sup_{y\in\Q,y<x}z_t^{(s,y)},\qquad
(s,x)\in E, t\in\Q, t>s.
\end{equation}
Denote by $D_E^{*,-}$ the set of those $z\in D_E$ where conditions
(\ref
{ZDEG1}), (\ref{ZFLOW}), (\ref{ZCADLAG}) and
(\ref{ZLEFT}) all hold, and define $\Phi$ on $D^{*,-}_E$ exactly as on
$D^{*,+}_E$.
Then, by a similar argument, $D_E^{\circ,-}=D^{*,-}_E$ and
$Z^{E,-}\dvtx D^\circ(\R,\cD)\to D_E^{\circ,-}$ is a bijection with
inverse \mbox{$\Phi^{E,-}=\Phi$}.
In particular,
$\Phi^{E,+}=\Phi^{E,-}$ on $D_E^{\circ,-}\cap D_E^{\circ,+}$ and so
$D_E^\circ=D_E^{\circ,-}\cap D_E^{\circ,+}$, as claimed.
\end{pf}

%
\begin{propositionn}\label{MPI}
Let $E$ be a countable subset of $\R^2$ containing $\Q^2$.
Then $\mu_E(C^\circ_E)=1$.
\end{propositionn}

\begin{pf}
We use an identification of $C^\circ_E$ analogous to that implied for
$D^\circ_E$ by the preceding proof.
The same five conditions (\ref{ZDEG1}), (\ref{ZRIGHT}), (\ref{ZFLOW}),
(\ref{ZCADLAG}) and (\ref{ZLEFT}) characterize
$C^\circ_E$ inside $C_E$, except that, in (\ref{ZCADLAG}), only the
case $m=0$ is allowed.
Recall that, under $\mu_E$, for time--space starting points $e=(s,x)$
and $e'=(s',x')$,
the coordinate processes $Z^e$ and $Z^{e'}$ behave as independent
Brownian motions up to
\[
T^{ee'}=\inf\bigl\{t\ge s\vee s'\dvtx
Z_t^e-Z_t^{e'}\in\Z\bigr\},
\]
after which they continue to move as Brownian motions, but now with a
constant separation.
In particular, if $s=s'$ and $x'=x+n$ for some $n\in\Z$, then
$T^{ee'}=0$, so $Z^{e'}_t=Z^e_t+n$
for all $t\ge s$, so (\ref{ZDEG1}) holds almost surely.

Let $(s,x)\in E$ and $t,u\in\Q$, with $s\le t<u$.
Consider the event
\[
A= \Bigl\{\sup_{y\in\Q,y<Z_t^{(s,x)}}Z_u^{(t,y)}=Z_u^{(s,x)}=
\inf_{y'\in
\Q,y'>Z_t^{(s,x)}}Z_u^{(t,y')} \Bigr\}.
\]
Fix $n\in\N$ and set $Y=n^{-1}\lfloor nZ_t^{(s,x)}\rfloor$ and $Y'=Y+1/n$.
Then $Y$ and $Y'$ are $\cF_t$-measurable, $\Q$-valued random variables.
Now $\PP(Y<Z_t^{(s,x)}<Y')=1$ and
\[
\bigl\{Y<Z_t^{(s,x)}<Y'\bigr\}\cap\bigl
\{T^{(t,Y)(t,Y')}\le u\bigr\}\sse A.
\]
By the Markov property of Brownian motion, almost surely,
\[
\PP\bigl(T^{(t,Y)(t,Y')}\le u|\cF_t\bigr)\ge2\Phi\biggl(
\frac{1}{n\sqrt{2(u-t)}} \biggr),
\]
and the right-hand side tends to $1$ as $n\to\infty$.
So, by bounded convergence, we obtain $\PP(A)=1$.
On taking a countable intersection of such sets $A$ over the possible values
of $s,x,t$ and $u$, we deduce that
conditions (\ref{ZRIGHT}), (\ref{ZFLOW}) and (\ref{ZLEFT}) hold
almost surely.

It remains to establish the continuity condition (\ref{ZCADLAG}).
For a standard Brownian motion $B$ starting from $0$, we have, for
$n\ge4$,
\[
\PP\Bigl(\sup_{t\le1}|B_t|>n \Bigr)\le
e^{-n^2/2}.
\]
Define, for ${\delta}>0$ and $e=(s,x)\in E$,
\[
V^e({\delta})=\sup_{s\le t\le s+{\delta}^2}\bigl|Z^e_t-x\bigr|.
\]
Then, by scaling,
\[
\PP\bigl(V^e({\delta})>n{\delta}\bigr)\le e^{-n^2/2}.
\]
Consider, for each $n\in\N$ the set
\[
E_n=\bigl\{\bigl(j2^{-2n},k2^{-n}\bigr)\dvtx j
\in\tfrac12\Z\cap\bigl[-2^{2n},2^{2n}\bigr),k=0,1,\ldots
,2^n-1\bigr\}
\]
and the event
\[
A_n=\bigcup_{e\in E_n}\bigl
\{V^e\bigl(2^{-n}\bigr)>n2^{-n}\bigr\}.
\]
Then $\PP(A_n)\le|E_n|e^{-n^2/2}$, so $\sum_n\PP(A_n)<\infty$, so by
Borel--Cantelli, almost surely,
there is a random $N<\infty$ such that $V^e(2^{-n})\le n2^{-n}$ for all
$e\in E_n$, for all $n\ge N$.

Given $\ve>0$, choose $n\ge N$ such that $(4n+2)2^{-n}\le\ve$ and set
${\delta}=2^{-2n-1}$. Then, for all
rationals $s,t\in(-n,n)$ with $0<t-s<{\delta}$ and all rationals
$x\in[0,1]$,
there exist $e^\pm=(r,y^\pm)\in E_n$ such that
\begin{eqnarray*}
 r&\le& s<t\le r+2^{-2n},
\\
 x+n2^{-n}&<&y^+\le x+(n+1)2^{-n},
\\
 x-(n+1)2^{-n}&\le& y^-<x-n2^{-n}.
\end{eqnarray*}
Then, $Z^{e^-}_s<x<Z_s^{e^+}$, so
\[
x-\ve\le Z^{e^-}_t\le Z_t^{(s,x)}\le
Z_t^{e^+}\le x+\ve.
\]
Hence, $\|\Phi_{(s,t]}-\id\|\le\ve$, as required.
\end{pf}

Recall that $\Phi^E$ denotes the inverse of the evaluation map $Z^E
\dvtx D^\circ(E) \to D^\circ_E$.

%
\begin{propositionn}
\label{PCP}
Let $E$ be a countable subset of $\R^2$ containing $\Q^2$. Then $\Phi
^E$ is continuous.
\end{propositionn}

\begin{pf}
Consider a sequence $(z_k\dvtx k\in\N)$ in $D^\circ_E$ and suppose that
$z_k\to z$ in $D_E$, with $z\in D^\circ_E$.
Set $\phi^k=\Phi^E(z_k)$ and $\phi=\Phi^E(z)$. By analogy with the
standard Skorohod topology,
it will suffice to show that, for all $n_0\in\N$ and all continuity
points $-n_0<t<n_0$, that is, $\phi_{\{t\}}=\id$, we have $\sup
_{-n_0<s<t} d_\cD(\phi_{(s,t]}^k,\break\phi_{(s,t]})\to0$ as $k\to\infty$.
Given $\ve>0$, choose $0< \eta< \ve/3$. As in the proof of
separability in Proposition~\ref{compsep}, there exist $m,n \in\N$ and
discontinuity points $-n_0=u_0< u_1 < \cdots< u_n=n_0$ with $2/m + 3
\eta< \ve$ such that if $I \cap\{u_0, \ldots, u_n \}$ with $\sup I -
\inf I < 2/m$, then $\|\phi_I-\id\|<\eta$. Consider the finite set
\[
F=\bigl(m^{-1}\Z\cap[-n_0,n_0]\bigr)\times
\bigl(m^{-1}\Z\cap[0,1)\bigr).
\]
There exists a $K<\infty$ such that, for all $k\ge K$ and all
$e_0=(s_0,x_0)\in F$, $d_{e_0}(z_k^{e_0}, z^{e_0}) < 1/m$.
Therefore, there exists some homomorphism of $(s_0, \infty)$, $\lambda
(= \lambda_{k, e_0})$, such that for all $t\in(s_0,n_0]|\lambda
(t)-t|<1/m$ and
\[
\bigl|\phi^{k,+}_{(s_0,t]}(x_0)-\phi_{(s_0, \lambda(t)]}^+(x_0)\bigr|=\bigl|
\phi^{k,-}_{(s_0,t]}(x_0)-\phi_{(s_0, \lambda(t)]}^-(x_0)\bigr|<1/m.
\]
For all $s\in[-n_0,n_0)$ and all $x\in[0,1)$, there exists
$(s_0,x_0)\in F$ such that
\[
s_0\le s<s_0+1/m, \qquad x_0\le x+
\eta+2/m<x_0+1/m.
\]
Then
\[
\phi^{k,+}_{(s_0,s]}(x_0)\ge\phi^+_{(s_0,\lambda(s)]}(x_0)-1/m
\ge x_0-\eta-1/m>x,
\]
so
\[
\phi^{k,+}_{(s_0,t]}(x_0)\ge\phi^{k,+}_{(s,t]}(x),\qquad
t\ge s.
\]
Now, for all $t\in(s,n_0]$ with $|t - u_l| > 1/m$ for all $l \in\{0,
\ldots, n\}$,
we have $d_\cD(\phi_{(s_0, \lambda(t)]}, \phi_{(s,t]}) < 2 \eta$, so
\[
\phi^+_{(s_0, \lambda(t)]}(x_0) \leq\phi^+_{(s,t]}(x_0
+ 2 \eta) + 2 \eta.
\]
So,
\[
\phi_{(s_0,t]}^{k,+}(x_0)\le\phi_{(s_0, \lambda(t)]}^+(x_0)+1/m
\le\phi^+_{(s,t]}(x_0 + 2 \eta) + 2\eta+ 1/m,
\]
and so
\[
\phi_{(s,t]}^{k,+}(x)\le\phi_{(s,t]}^+(x+\ve)+\ve.
\]
By a similar argument, for all $t\in(s,n_0]$ with $|t - u_l| > 1/m$ for
all $l \in\{0, \ldots, n\}$,
\[
\phi_{(s,t]}^{k,-}(x)\ge\phi_{(s,t]}^-(x-\ve)-\ve,
\]
so $d_\cD(\phi^k_{(s,t]},\phi_{(s,t]})\le\ve$. As $1/m$ can be chosen
to be arbitrarily small, the result follows.
\end{pf}

\subsection{List of notation}
For ease of reference, we list below some of the notation that appears
in the paper. In all definitions, $e=(s,x)\in\R^2$,
$E=(e_k=(s_k,x_k)\dvtx k\in\N)$ in $\R^2$, $\ve\in(0,1]$ and disturbance
flows are with disturbance~$f$. \vspace*{15pt}

{\hspace*{-14pt}
\begin{tabular*}{\textwidth}{@{\extracolsep{\fill}}p{2.6cm}p{9cm}@{}}
&{\textbf{Disturbance flows}}: \\
$\Phi_{n,m}$: & The discrete disturbance flow in which disturbances are
applied at integer times.\\
$\Phi$:& The lattice disturbance flow, in which the disturbances are
applied at times in the lattice $\Z/\rho$, or the Poisson disturbance
flow, in which the disturbances are applied at the times of the atoms
of a Poisson process with intensity $\rho$. \\
$\hat\Phi$:& The time reversed disturbance flow given by $\hat\Phi
_I=\Phi_{-I}^{-1}$. \\
$\Phi^\ve$:& The $\ve$-scale disturbance flow, that is, $\Phi
_I^\ve=\s_\ve
(\Phi_{\ve^2 I})$.\\
$Z^{e,\pm}$:& The evaluation maps $Z^{e,\pm}\dvtx D^\circ(\R,\cD)
\to D_e$
given by $Z^{e,\pm}_t(\phi)=\phi_{(s,t]}^\pm(x)$.\\
$Z^{E,\pm}$:& The evaluation maps $Z^{E,\pm}\dvtx D^\circ(\R,\cD
)\to D_E$
given by $Z^{E,\pm}(\phi)=(Z^{e_k,\pm}(\phi)\dvtx k\in\N)$. \\
$\check Z^{e,\pm}(\phi)$:& The extension of the evaluation maps from
$[s,\infty)$ to the whole of $\R$.\\
\end{tabular*}}

{\hspace*{-14pt}
\begin{tabular*}{\textwidth}{@{\extracolsep{\fill}}p{2.6cm}p{9cm}@{}}
$\Phi^{E,\pm}$:& The inverse of $Z^{E,\pm}$ restricted to $D^{\circ
,\pm
}_E$. \\
$\Phi^{E}$:& The inverse of $Z^{E,+}$ (or identically $Z^{E,-}$)
restricted to $D^{\circ}_E$.\\\\

&{\textbf{Metric spaces}}: \\
$(\cD,d_\cD)$:& The set of disturbances on the circle together with
the metric defined in \eqref{distdist}. \\
$(\bar\cD,d_{\bar\cD})$:& The space of disturbances on the line
together with the metric defined in \eqref{ddbardef}.\\
$\cD^*$:& $\cD^*= \{f\in\cD\sm\{\id\}\dvtx\int_0^1(f(x)-x) \,
dx=0 \}
$. \\
$(D(\R, S), d)$:& The Skorohod space of cadlag paths in a metric space
$S$, equipped with $d$, the Skorokhod metric on $D(\R, S)$. \\
$(D_e, d_e)$:& $D_e=D_x([s,\infty),\R)$ is the Skorohod space of
cadlag paths starting from $x$ at time $s$, equipped with $d_e$, the
Skorokhod metric on $D_e$.
\\
%
$(D_E,d_E)$:& $D_E=\prod_{k=1}^\infty D_{e_k}$ and $d_E$ is the metric
on $D_E$ defined in \eqref{DEmetric}. \\
$\check D_e$:& $\check D_e=\{\xi\in D(\R,\R)\dvtx\xi_s=x\}$. \\
$\check D_E$:& $\check D_E=\prod_{k=1}^\infty\check D_{e_k}$. \\
$C_e$:& The subspace of $D_e$ consisting of continuous paths.\\
$C_E$:& The subspace of $D_E$ where each coordinate path is continuous,
that is, $C_E=\prod_{k=1}^\infty C_{e_k}$.\\
$(C^\circ(\R,\cD), d_C)$:& The set of continuous weak flows on the
circle with values in $\cD$ together with the metric defined in \eqref
{dcdist}. \\
$C^\circ(\R,\bar\cD)$:& The space of continuous weak flows on the
line. \\
\end{tabular*}}

\hspace*{-14pt}
\begin{tabular*}{\textwidth}{@{\extracolsep{\fill}}p{2.6cm}p{9cm}@{}}
$d_C^{(n)}$:& The semimetric on $C^\circ(\R,\cD)$, restricted to time
taking values in $(-n,n)$, as defined in \eqref{dcndist}. \\
$C^{\circ,\pm}_E$:& The subspace of $C_E$ given by $C^{\circ,\pm
}_E=\{
Z^{E,\pm}(\phi)\dvtx\phi\in C^\circ(\R,\cD)\}$. \\
$C^\circ(E)$:& $C^\circ(E)=\{\phi\in C^\circ(\R,\cD)\dvtx
Z^{E,+}(\phi
)=Z^{E,-}(\phi)\}$. \\
$C^\circ_E$: & $C^\circ_E=\{Z^E(\phi)\dvtx\phi\in C^\circ(E)\}$.
\\
$(D^\circ(\R,\cD), d_D)$:& The set of cadlag weak flows with values in
$\cD$ together with the metric defined in \eqref{SMET}. \\
$D^\circ(\R,\bar\cD)$:& The space of cadlag weak flows on the
line. \\
$d_D^{(n)}$:& The semimetric on $D^\circ(\R,\cD)$, on a restricted
time-interval, as defined in \eqref{ddndist}. \\
$D^{\circ,\pm}_E$:& The subspace of $D_E$ given by $D^{\circ,\pm
}_E=\{
Z^{E,\pm}(\phi)\dvtx\phi\in D^\circ(\R,\cD)\}$. \\
$D^\circ(E)$:& $D^\circ(E)=\{\phi\in D^\circ(\R,\cD)\dvtx
Z^{E,+}(\phi
)=Z^{E,-}(\phi)\}$. \\
$D^\circ_E$:& $D^\circ_E=\{Z^E(\phi)\dvtx\phi\in D^\circ(E)\}$.\\\\
\end{tabular*}

\hspace*{-14pt}
\begin{tabular*}{\textwidth}{@{\extracolsep{\fill}}p{2.6cm}p{9cm}@{}}
&{\textbf{Distributions}}: \\\
$\mu_e$:& The distribution on the Skorohod space $D_e$ of a standard
Brownian motion starting from $e$.\\
$\mu^f_e$:& The distribution on $D_e$ of the process $(\Phi
_{(s,t]}(x))_{t \geq s}$. \\
$\mu_E$, $\bar\mu_E$:& The distribution on $D_E$ (or $C_E$) of a
sequence of coalescing Brownian motions on the circle, respectively on
the line, starting from $E$.\\
$\mu^f_E$, $\mu^{f,\ve}_E$:& The distributions on $D_E$ of $(\Phi
_{(s_k, \cdot]}(x_k)\dvtx k\in\N)$, $(\Phi^\ve_{(s_k, \cdot
]}(x_k)\dvtx k\in\N
)$, respectively. \\
$\mu_A$, $\bar\mu_A$:& The distribution on $C^\circ(\R,\cD)$,
respectively on $C^\circ(\R,\bar\cD)$, of the coalescing Brownian flow
on the circle, respectively on the line. \\
$\mu^f_A$, $\hat\mu_A^f$:& The distributions on $D^\circ(\R,\cD
)$ of
$\Phi$, $\hat\Phi$, respectively. \\
$\check\mu_E^f$:& The law on $\check D_E$ of $(\check Z^{e_k}\dvtx
k\in\N)$
under $\mu_A^f$. \\
$\check\mu_E$:& The law on $\check D_E$ of $(\check Z^{e_k}\dvtx
k\in\N)$
under $\mu_A$. \\
\end{tabular*}
\end{appendix}
\section*{Acknowledgements}
We are grateful to Tom Ellis for helpful suggestions during the writing
of this paper. We would also like to thank an anonymous referee for a
careful reading of our manuscript and many useful comments.

%
%

%




\printaddresses

\end{document}